\documentclass[10pt]{article}
\usepackage{amsthm,amsfonts,amsmath,amscd,amssymb,verbatim}
\usepackage{graphicx}

\title{Stable Hamiltonian structures in dimension three are supported
  by open books}
\author{K.~Cieliebak and E.~Volkov\footnote{Supported by DFG
    grant CI 45/5-1, FNRS grant FRFC 2.4655.06, and the ESF
    Research Networking Programme {\em Contact and Symplectic Topology
      (CAST)}}}
\theoremstyle{plain}
\newtheorem{theorem}{Theorem}[section]

\newtheorem{corollary}[theorem]{Corollary}
\newtheorem{cor}[theorem]{Corollary}
\newtheorem{proposition}[theorem]{Proposition}
\newtheorem{prop}[theorem]{Proposition}
\newtheorem{lemma}[theorem]{Lemma}

\theoremstyle{remark}

\newtheorem{remark}[theorem]{Remark}
\newtheorem{example}[theorem]{Example}

\newcommand{\id}{{{\mathchoice {\rm 1\mskip-4mu l} {\rm 1\mskip-4mu l}
{\rm 1\mskip-4.5mu l} {\rm 1\mskip-5mu l}}}}

\newcommand{\p}{\partial}

\newcommand{\om}{\omega}

\newcommand{\eps}{\varepsilon}

\newcommand{\la}{\langle}
\newcommand{\ra}{\rangle}
\newcommand{\N}{{\mathbb{N}}}
\newcommand{\Z}{{\mathbb{Z}}}
\newcommand{\R}{{\mathbb{R}}}
\newcommand{\C}{{\mathbb{C}}}

\newcommand{\im}{{\rm im }}        
\newcommand{\st}{{\rm st}}

\newcommand{\const}{{\rm const}}

\renewcommand{\min}{{\rm min}}

\newcommand{\inn}{{\rm int\,}}

\newcommand{\LL}{\mathcal{L}}
\newcommand{\FF}{\mathcal{F}}

\newcommand{\PP}{\mathcal{P}}

\begin{document}

\maketitle

\section{Introduction}\label{sec:intro}

In dimension three, (cooriented) contact structures are closely
related to open books via the following results, the first due to
Thurston-Winkelnkemper~\cite{TW} and the other ones due to
Giroux~\cite{Gir}:

(1) Every open book supports a contact structure.

(2) Any two contact structures supported by the same open book are
  connected by a contact isotopy supported by the open book.

(3) Every contact structure is supported by an open book. 

(4)  Two open books supporting the same contact structure are
  isotopic after finitely many stabilizations. 

In~\cite{CV} we proved analogues of the first two results for stable
Hamiltonian structures in dimension three:

(1') Every open book supports a stable Hamiltonian structure
  realizing a given cohomology class and given signs at the
binding components.     

(2') Any two stable Hamiltonian structures supported by the same open
book in the same cohomology class and with the same signs at
the binding components are connected by a stable homotopy supported by
the open book. 

In this paper we prove the analogue of the third result.  

\begin{theorem}\label{thm:Giroux}
Every stable Hamiltonian structure on a closed oriented $3$-manifold
is stably homotopic to one which is supported by an open book.
\end{theorem}

{\bf Definitions. }
Let us first explain the notions appearing in the statement;
see~\cite{CV} for more details and background. 

Let $M$ denote a closed oriented $3$-manifold. A {\em stable
  Hamiltonian structure (SHS)} on $M$ is a pair $(\om,\lambda)$
consisting of a closed 2-form $\om$ and a 1-form $\lambda$ such that 
\begin{equation}\label{eq:def}
   \lambda\wedge\om>0 \qquad\text{and}\qquad  
   \ker(\om)\subset\ker(d\lambda).
\end{equation}
It induces a canonical {\em Reeb vector field} $R$
generating $\ker(\om)$ and normalized by $\lambda(R)=1$. 
Note that in dimension 3 the second condition in~\eqref{eq:def} is   
equivalent to 
$$
   d\lambda=f\om
$$
for some $f\in C^{\infty}(M)$ which we write as $f=d\lambda/\om$.
A {\em stable homotopy} is a smooth family of SHS $(\om_t,\lambda_t)$
{\em such that the cohomology class of $\om_t$ remains constant}. (The
condition on the cohomology class is natural for many reasons, e.g.~to
get invariance of symplectic field theory under stable
homotopies). A SHS $(\om,\lambda)$ with $0=[\om]\in H^2(M;\R)$ is
called {\em exact}. 
Note that each contact form $\lambda$ (i.e.~satisfying
$\lambda\wedge d\lambda>0$) induces an exact stable Hamiltonian
structure $(d\lambda,\lambda)$. 

An {\em open book decomposition} of $M$ is a pair $(B,\pi)$, where 
$B\subset M$ is an oriented link (called the {\em binding}) and  
$\pi:M\setminus B\rightarrow S^1$ is a fibration 
satisfying the following condition near $B$: Each connected
component $B_l$ of $B$ has a tubular neighbourhood $S^1\times
D^2\hookrightarrow M$ with orienting coordinates 
$(\phi,r,\theta)$, where $\phi$ is an orienting coordinate along
$S^1\cong B_l$ and $(r,\theta)$ are polar coordinates on
$D^2$, such that on $(S^1\times D^2)\setminus B_l$ we have
$\pi(\phi,r,\theta)=\theta$. It follows that the closure of each fibre  
$\pi^{-1}(\theta)$ is an embedded compact oriented surface
$\Sigma_\theta$ (called a {\em page}) with boundary $B$.  

We call a SHS $(\om,\lambda)$ {\em supported by the open book
  $(B,\pi)$} if $\om$ is positive on each fibre $\pi^{-1}(\theta)$. It
follows that $\lambda$ is nowhere vanishing on the binding. Thus for
every binding component $B_l$ we have a sign $s(B_l)$ which is $+1$
iff $\lambda$ induces the orientation of $B_l$ as boundary of a page. 
We emphasize that this {\em differs from the common notion in contact
  topology} where one requires that all signs are $+1$. The results
(i-iv) above in the contact case have to be understood for this more
restrictive notion to which we will refer as {\em positively
  supported}. It is shown in~\cite{CV} that in
Theorem~\ref{thm:Giroux} we cannot achieve ``positively supported'': 
There exists a stable Hamiltonian structure on $S^3$ which is not
stably homotopic to one that is positively supported by an open book. 

On the other hand, if an exact SHS is positively supported by an open
book, then by results (1) and (2') above it is stably homotopic to a
positive contact structure (i.e.~to a SHS of the form
$(d\lambda,\lambda)$ for a positive contact form $\lambda$). 
 
\begin{example}\label{ex:confol}
Consider an exact SHS $(\om,\lambda)$ for which $\lambda$ defines a
{\em confoliation}, i.e.~$\lambda\wedge d\lambda\geq 0$
(see~\cite{ET}). Then $(\om,\lambda)$ is stably homotopic to a
positive contact structure. To see this, note that 
the confoliation condition is equivalent to $f=d\lambda/\om\geq
0$. This allows us to achieve positive signs at all binding components
of the open book in Theorem~\ref{thm:Giroux} (see
Remark~\ref{rem:positive}), so the resulting SHS  
is positively supported by the open book and hence stably
homotopic to a positive contact structure. The stable homotopy can
also be constructed explicitly as follows: Write $\om=d\alpha$. Then
$\lambda_t:=\lambda+t\alpha$, $t\in[0,\eps]$, defines for small
$\eps>0$ a homotopy of stabilizing forms from $\lambda$ to the
positive contact form $\lambda_\eps$, and
$\bigl((1-t)\om+t\,d\lambda_\eps,\lambda_\eps\bigr)$ yields a stable
homotopy from $(\om,\lambda_\eps)$ to $(d\lambda_\eps,\lambda_\eps)$. 
\end{example}

\medskip

{\bf Sketch of proof. }
The proof of Theorem~\ref{thm:Giroux} departs from a structure
theorem proved in~\cite{CV} (see Section~\ref{sec:structure}): For
each SHS $(\om,\lambda)$ on a closed 
oriented 3-manifold $M$ we can change the 1-form $\lambda$ such that
$M=\bigcup_iN_i\cup\bigcup_jU_j$ is a union of compact regions such
that $(\om,\lambda)$ is $T^2$-invariant on $U_j\cong[0,1]\times T^2$,
and $d\lambda=c_i\om$ on $N_i$ with constants
$c_i\in\R$. We refer to the regions $U_j$ as {\em integrable regions}, 
and to the regions $N_i$ with $c_i=0$ (resp.~$>0$, $<0$) as
{\em flat (resp.~positive / negative contact) regions}. 

On a flat region we perturb and rescale $\lambda$ to make it integral
and obtain a fibration over $S^1$. On a (positive or negative) contact
region we use a relative version of Giroux's existence theorem
(3) above to produce an open book supporting the contact form
$\lambda$ (and hence the SHS $(\om,\lambda)$) which induces a
fibration $T^2\to S^1$ on each boundary torus. Finally, we use
techniques from~\cite{CV} to extend the SHS and open books over the
integrable regions $U_j$ to a SHS and supporting open book on $M$
(Section~\ref{sec:proof}).  

To prove the relative version of (3) we collapse a circle direction
transverse to the Reeb direction $\p N_i$ to obtain
a closed contact manifold $(\bar N_i,\bar\lambda)$
(Section~\ref{sec:relGiroux}). 
Each boundary torus $T_j$ gives rise to a transverse knot $L_j$. We
use Giroux's existence theorem (3) to find a supporting open book for
$(\bar N_i,\bar\lambda)$ and apply a result of Pavelescu~\cite{Pav} to
braid the link $L=\cup L_j$ around its binding
(Section~\ref{sec:braiding}). After standardizing $\bar\lambda$ near
the resulting link (Section~\ref{sec:standard}) we replace its
components back by 2-tori to obtain the desired open book on $N_i$. 
\medskip

{\bf Acknowledgements. }
We thank Y.~Eliashberg and J.~Etnyre for fruitful discussions.

\section{Braiding transverse knots around the binding of 
a contact open book}\label{sec:braiding}

We will use the following terminology. A contact form $\alpha$ is {\em
  supported by an open book $(B,\pi)$} if $\alpha$ restricts
positively to the binding and $d\alpha$ restricts positively to the
pages. A contact structure $\xi$ is {\em supported by an open book} if there
exists a contact form defining $\xi$ which is supported by the open
book. An oriented link $L$ is {\em (positively) transverse to the
  contact structure $\xi$} if a defining contact form $\alpha$
restricts positively to $L$. An oriented link $L$ is {\em braided
  around an open book $(B,\pi)$} if $L$ is disjoint from the binding
and positively transverse to the pages, i.e.~$d\pi$ restricts
positively to $L$. 

The goal of this section is to explain the following result from
E.~Pavelescu's thesis~\cite{Pav}. 

\begin{theorem}
\label{thm:braiding}
Let $M$ be a closed oriented $3$-manifold and $\xi$ be a cooriented
contact structure on $M$ supported by an open book $(B,\pi)$, and 
$L\subset M$ be a link positively transverse to $\xi$. 
Then there exist isotopies
$(B_t,\pi_t,L_t)_{t\in[0,1]}$ of open books $(B_t,\pi_t)$ supporting
$\xi$ and links $L_t$ transverse to $\xi$ such that
$(B_0,\pi_0,L_0)=(B,\pi,L)$ and $L_1$ is braided around $(B,\pi)$. 

Moreover, for every collection of suficiently large natural numbers
$k_1,\dots,k_\ell\geq K$, where $\ell$ is the number of components of
$L$ and $K$ a constant depending on $(B,\pi,\xi,L)$, we can arrange
that the intersection number of the $i$-th component of $L_1$ with a
page of $(B_1,\pi_1)$ equals $k_i$. 
\end{theorem}

\begin{remark}
Pavelescu claims the stronger result that the open book
$(B_t,\pi_t)=(B,\pi)$ can be fixed. As the proof in~\cite{Pav}
contains some gaps, we repeat it below with some more details. 
The deformation of the open book is needed 
for technical reasons; it can be made $C^1$-small, and presumably be
avoided with more work. 
\end{remark}

The proof in~\cite{Pav} is based on the following construction.  
Consider a contact form $\alpha$ supported by an open book
$(B,\pi)$. Suppose that 
$$
   \alpha=T(1-r^2)(d\phi+r^2d\theta)
$$
in a tubular neighbourhood $W=\{r\leq r_0\}$ of $B$, where $T>0$ is
some constant and $(\phi,r,\theta)$ are (polar) coordinates near the
respective binding components in which $\pi=\theta$. 
We define another neighbourhood 
$$
  V:=\{r\le r_0/2\}\subset\{r\le r_0\}=W
$$ 
of $B$. Let $f:[0,r_0]\to[0,1]$ be a nondecreasing function which
equals $T(1-r^2)r^2$ for $r\leq r_0/2$ and $1$ near $r=r_0$. It
extends by $1$ over $M\setminus W$ to a function on $M$ that we also
denote by $f$. Consider the family of contact structures 
\begin{equation}\label{eq:defcontflow}
  \xi_t:=\ker\alpha_t, \qquad \alpha_t:=\alpha+tf\,d\theta, \qquad
  t\in[0,\infty). 
\end{equation}

By Gray's stability theorem we know that there is a family of
diffeomorphisms $\{\Psi_t\}_{t\in [0,\infty)}$ such that $\Psi_t$
pulls back $\alpha_t$ to a multiple of $\alpha_0=\alpha$.  
We need to analyse $\{\Psi_t\}$ more closely. 
Recall that $\Psi_t$ is given as the flow of the time-dependent
vector field $X_t\in\xi_t$ defined by 
\begin{equation}\label{eq:Gray}
   i_{X_t}d\alpha_t = h_t\alpha_t-\dot\alpha_t,\qquad
   h_t:=\dot\alpha_t(R_t), 
\end{equation}
where $R_t$ denotes the Reeb vector field of $\alpha_t$ and
$\dot\alpha_t$ denotes the time derivative of $\alpha_t$. 

Consider a page $\Sigma$ and note that
$\alpha_t|_\Sigma=\alpha_\Sigma$, so all the $\alpha_t$ define the
{\em same} characteristic foliation $\FF=\xi\cap T\Sigma =
\xi_t\cap\Sigma$ on $\Sigma$. Since $d\alpha_t|_\Sigma=d\alpha_\Sigma$
is positive, $R_t$ is positively transverse to $\Sigma$ and 
$\dot\alpha_t=f\,d\theta$ yields $h_t=f\,d\theta(R_t)>0$. 
Contracting equation \eqref{eq:Gray} with any vector $v\in\xi\cap T\Sigma$
we obtatain $d\alpha(X_t,v)=0$, and since $v,X_t\in\xi$ and $d\alpha|_\xi$ is
nondegenerate this implies that $v$ and $X_t$ are collinear. This
shows that $X_t$ is tangent to the pages and spans the charateristic
foliation. Moreover, 
the restriction of equation~\eqref{eq:Gray} to $\Sigma$ gives
\begin{equation}\label{eq:Graypage}
   i_{X_t}d\alpha|_\Sigma=h_t\alpha|_\Sigma.
\end{equation}
Thus $X_t$ is determined by equation~\eqref{eq:Graypage}. In
particular, each $X_t=h_tX$ is a positive multiple of the Liouville
field $X$ tangent to the pages defined by 
$$
   i_Xd\alpha|_\Sigma=\alpha|_\Sigma.
$$
A short computation shows $X=-\frac{1-r^2}{2r}\p_r$ on $W$, so $X$
points into $V$ along $\p V$. 
The key property of $X$ is that $L_Xd\alpha|_\Sigma= 
d(i_Xd\alpha|_\Sigma)=d\alpha|_\Sigma$, i.e.~$X$ expands the positive
area form $d\alpha|_\Sigma$ on the page. 
This has the following dynamical consequences:
\begin{enumerate}
\item Each closed orbit of $X$ is repelling.
\item At each zero $p\in\Sigma$ of $X$ the linearization
$d_pX|_{T\Sigma}$ has an eigenvalue with positive real part. 
If the eigenvalues are non-real, or both real and positive, this
implies that $p$ is nondegenerate and $X$ flows out of $p$; we call
such $p$ {\em elliptic}.  
If the eigenvalues are real with one positive and one non-positive we
call $p$ {\em hyperbolic}; in this case there may be one or two
flow lines converging to $p$ in forward time. 
\item As a consequence of (i-ii) and the Poincar\'e-Bendixson Theorem,
every flow line of $X$ which is not a zero or a closed orbit either
enters $V$ in finite time or converges to a hyperbolic zero in forward
time. 
\end{enumerate}
Now consider the $S^1$-family of pages
$\Sigma_\theta:=\pi^{-1}(\theta)$, $\theta\in S^1$. For any $\theta\in
S^1$ let $S_\theta$ denote the subset of $\Sigma_\theta$ consisting of
zeroes, closed orbits, and flow lines converging to hyperbolic zeros
in forward time. We set 
\begin{equation}\label{eq:nogo}  
   S:=\cup_{\theta\in S^1}S_\theta
\end{equation}

The following statement is crucial for the proof.

\begin{lemma}\label{lem:escape}
Let $U_S$ be an open neighbourhood of $S$ in $M\setminus B$. 
Then there exists $\tau\in [0,+\infty)$ such that 
$\Psi_\tau(M\setminus U_S)\subset V$. 
\end{lemma}

\begin{proof}
By property (iii) above and compactness, there exists a constant
$\tau_0>0$ such that each point in $M\setminus U_S$ reaches $V$ in
time $\tau_0$ under the flow of $X$.
To get the corresponding statement with the time-dependent 
vector field $X_t$ in place of $X$ we need to look at the behaviour of
the length of $X_t$ as $t\to +\infty$. 
Since the $d\theta$ component of $\alpha_t$
goes to $\infty$ as $t\to 0$, we get that $R_t\to 0$ as $t\to\infty$.  
>From 
$$
   1 = \alpha_t(R_t) = \alpha(R_t) + tf\,d\theta(R_t)
$$
and $\alpha(R_t)\to 0$ as $t\to\infty$ we see that $h_t =
f\,d\theta(R_t) =  O(t^{-1})$ as $t\to \infty$ on $M\setminus B$. 
Since $\int_1^{\infty}\frac{1}{t}dt=\infty$, there exists a constant
$\tau>0$ such that each point in $M\setminus U_S$ reaches $V$ in
time $\tau$ under the flow $\Psi_t$ generated by $X_t$.
\end{proof}

Besides the preceding discussion, we will also use in the proof the
following 3 lemmas which correspond to Lemmas 4.9, 4.10 and 4.12 in~\cite{CV}, 
respectively. 

\begin{lemma}\label{lm:cutoff}
For all $\delta,\varepsilon>0$ there exists a smooth function
$\rho:[0,\delta]\rightarrow[0,1]$ with the following properties:
$\rho$ is nonincreasing, constant $1$ in a neighbourhood of $0$,
constant $0$ in a neighbourhood of $\delta$, and 
$$
   |x\rho'(x)|<\varepsilon
$$ 
for all $x\in [0,\delta]$. 
\end{lemma}

\begin{lemma}\label{lem:obd-cont}
Let $\alpha$ be a positive contact form on $U=S^1\times D^2$ 
satisfying $d\theta\wedge d\alpha>0$ and
$T:=\int_{S^1\times\{0\}}\alpha>0$. Then 
there exists a homotopy rel $\p U$ of contact forms $\alpha_t$ 
satisfying $d\theta\wedge d\alpha_t>0$ such that $\alpha_0=\alpha$
and 
$$
  \alpha_1=T(r^2d\theta+(1-r^2)d\phi)
$$
near $S^1\times\{0\}$.
\end{lemma}

\begin{lemma}\label{lem:interpol2}
Let $\phi$ be the angular coordinate on
$S^1$ and $(r,\theta)$ be polar coordinates on $D^2$. 
Let $\sigma:[0,1]\to\R$ be a function, constant near $r=0$ and
supported in $[0,\delta]$, with $|r\sigma'(r)|\le \varepsilon$. 
Let $\alpha_0,\alpha_1$ be two contact forms on $S^1\times D^2$
satisfying $d\theta\wedge d\alpha_i\geq \beta>0$, $i=0,1$ and
$\alpha_1-\alpha_0=O(r^2)$ near $\{r=0\}$. Then for $\delta,
\varepsilon$ sufficiently small the 1-form $\alpha :=
(1-\sigma(r))\alpha_0 + \sigma(r)\alpha_1$ is contact and satisfies
$d\theta\wedge d\alpha>0$.  
\end{lemma}

After these preparations, we now turn to the 

\begin{proof}[Proof of Theorem~\ref{thm:braiding}]
{\bf Step 0. }
We first put the open book $(B,\pi)$ into nice position with respect
to $\xi$. 
After a small transverse isotopy (fixing $(B,\pi,\xi)$) we may assume
that the link $L$ does not intersect the binding $B$.  
Let $\alpha$ be a contact form defining $\xi$ and supported
by $(B,\pi)$.  
By Lemma~\ref{lem:obd-cont} we find a homotopy of contact forms
$\{\alpha_t\}_{t\in [0,1]}$ supported by $(B,\pi)$, fixed outside a
neighbourhood of $B$, with $\alpha_0=\alpha$ and  
$$
  \alpha_1|_W=T(r^2d\theta+(1-r^2)d\phi)
$$
for some neighbourhood $W=\{r\le r_0\}$ of $B$ and some locally
constant function $T>0$ on $W$, where $(\phi,r,\theta)$ are the
standard open book coordinates near $B$. 
Using Lemma~\ref{lem:interpol2}, we can further deform the
contact form $\alpha_1$, through contact forms supported by $(B,\pi)$
and fixed outside a neighbourhood of $B$, to a contact form
$\alpha_2$ satisfying  
$$
  \alpha_2|_W=T(1-r^2)(r^2d\theta+d\phi)
$$
(after shrinking $W$). Finally, we perturb $\alpha_2$,
keeping it fixed near $B$, to a contact form $\alpha_3$ for which the
charactistic foliations on the pages are sufficiently generic (in a
sense that is made precise below). 

After applying Gray's theorem, we may assume that $\xi$ and $L$ are
fixed and the open book is moving by an isotopy. We rename the new
open book back to $(B,\pi)$ and 
the new contact form $\alpha_3$ back to $\alpha$. 
\medskip

After this preparatory step, we will keep $(B,\pi,\xi)$ fixed and only
deform the transverse link $L$. More precisely, 
we will construct a family $(\xi_t,L_t)_{t\in[0,A]}$ of contact
structures $\xi_t$ and links $L_t$ in $M$ with the following
properties: 
\begin{description}
\item [(a)] there exists a family of diffeomorphisms $(\psi_t)_{t\in[0,A]}$
    with 
$$
   \psi_0=\id,\quad \psi_t^*\xi_t=\xi,\quad \psi_t(B)=B,\quad
   \pi\circ\psi_t=\pi; 
$$
\item [(b)] $L_0=L$ and $L_t$ is transverse to $\xi_t$ for all $t\in[0,A]$; 
\item [(c)] $L_A$ is braided around $(B,\pi)$ and the intersection number
  the $i$-th component of $L_A$ with a page is $k_i$. 
\end{description}
Then the links $L_t':=\psi_t^{-1}(L_t)$ satisfy $L_0'=L$, $L_t'$ is 
transverse to $\psi_t^*\xi_t=\xi$ for all $t\in[0,A]$, and
$L_A'$ is braided around
$\bigl(\psi_A^{-1}(B),\pi\circ\psi_A\bigr)=(B,\pi)$ with intersection
numbers $k_i$. Thus $L_t'$ is the desired isotopy.   

We construct the family $(\xi_t,L_t)_{t\in[0,A]}$ in 6 steps. 

{\bf Step 1. }
Recall that from Step 0 we have a contact form $\alpha$ defining $\xi$
supported by $(B,\pi)$ which is given by 
$$
   \alpha=T(1-r^2)(d\phi+r^2d\theta)
$$
in a tubular neighbourhood $W=\{r\leq r_0\}$ of $B$ not meeting $L$.
Let us call an arc $\gamma$ of $L$ {\it good} if $\gamma$ is
positively transverse to the pages, and {\em bad} otherwise. Let $L_g$
denote the union of all good arcs and set 
$L_b:=L\setminus L_g$. Since transversality is an open condition,
we know that $L_g$ is a union of open arcs and (after a small isotopy
of $L$) we may assume that $L_b$ is a union of closed arcs. Our
objective is to achieve $L_b\cap S=\emptyset$, so that we can use
Lemma~\ref{lem:escape} to push $L_b$ into a smaller neighbourhood 
$$
   V:=\{r\leq r_0/2\}
$$ 
of $B$. 
 
By the last generic perturbation in Step 0 we can achieve that the set
$S$ is a union of submanifolds of dimensions $0,1,2$ of the following
types: 
\begin{description}
\item [S(1)] finitely many zeroes of birth-death type;
\item [S(2)] nondegenerate (elliptic or hyperbolic) zeroes varying in
  1-dimensional families with $\theta$;
\item [S(3)] isolated flow lines converging in forward time to birth-death
  type zeroes;
\item [S(4)] isolated flow lines converging in forward and backward time to
  nondegenerate hyperbolic zeroes;
\item [S(5)] nondegenerate closed orbits varying in
  1-dimensional families with $\theta$;
\item [S(6)] flow lines converging in forward time to nondegenerate
  hyperbolic zeroes (and in backward time to elliptic zeroes or closed
  orbits) varying in 1-dimensional families with $\theta$. 
\end{description}

Note that, for each stratum $S(n)$, $\overline{S(n)}\setminus S(n)$ is
contained in the union
of strata $S(m)$ with $m<n$. We will successively make $L_b$ disjoint
from $S(1),\dots,S(6)$. To begin, we make $L_b$ disjoint from the
$0$- and $1$-dimensional strata $S(1),\dots,S(4)$ simply by a small
perturbation of $L$. 

{\bf Step 2. } 
Let $\PP=S(5)$ be the union of closed orbits of $X$. Each connected component 
of $\PP$ is an embedded $2$-torus in $M$ fibered over $S^1$ by $\pi$.
Consider a point $p\in\PP$ belonging to a page $\Sigma$. Note that the
three planes $T_p\Sigma$, $\xi_p$ and $T_p\PP$ all contain the line
$\R X(p)$, so their intersection with a plane $E\subset T_pM$
transverse to $X(p)$ gives three lines 
in $E$.

Recall that the contact structures $\xi_t=(\Psi_t)_*\xi$ defined
above (see equation \eqref{eq:defcontflow}) 
converge at $p$ to $T_p\Sigma$ as $t\to\infty$. After applying the
homotopy $\bigl(\xi_t,L_t=\Psi_t(L)\bigr)$, $t\in[0,\tau]$, for
sufficiently large $\tau$ and renaming $(\xi_\tau,L_\tau)$ back to
$(\xi,L)$, we may hence assume that at all points $p\in\PP$ the three
lines corresponding to $\Sigma,\xi,\PP$ (and their coorientations) are
ordered as in Figure~\ref{figure:wrinkling}).

\begin{figure}
\centering
\includegraphics{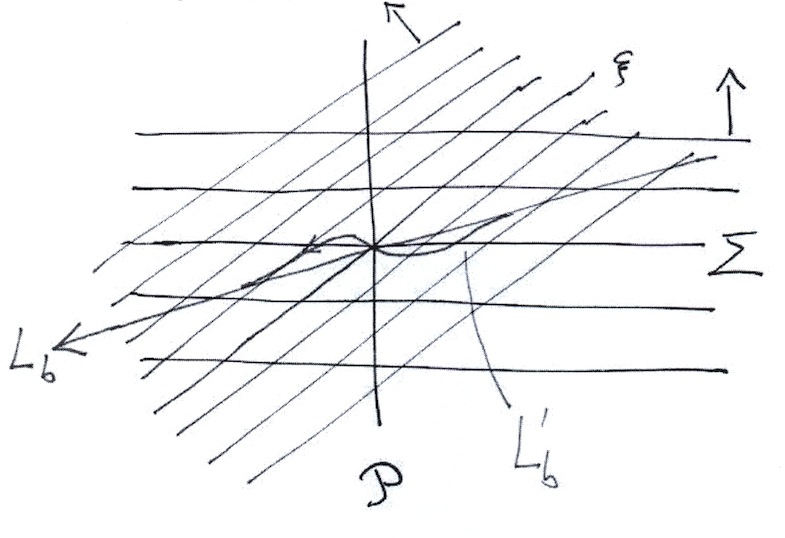}
\caption{Wrinkling}
\label{figure:wrinkling}
\end{figure}

{\bf Step 3. }
After a perturbation of $L$ we may assume that the intersection
$L_b\cap \PP$ is transverse. Consider a point $p\in L_b\cap \PP$. 
Pick local coordinates $(x,y,z)\in\R^3$ near $p=(0,0,0)$ in which
$\pi(x,y,z)=z$ and $\PP=\{y=0\}$, so $X(p)$ points in the
$x$-direction. By Step 2, the intersections of $T_p\Sigma$, $\xi_p$
and $T_p\PP$ with the plane $\{x=0\}$ are ordered as in Figure
\ref{figure:wrinkling}.  
After a further coordinate change near $p$ we may
assume that near $p$ the curve $L_b$ is a straight line segment
contained in the plane $\{x=0\}$. After zooming in near $0$ and
rescaling, we may assume that 
the contact planes in the neighbourhood are $C^0$-close to
$\xi_p$. Now we modify $L_b$ near $p$ within the plane $\{x=0\}$ as
shown in Figure \ref{figure:wrinkling}. The new arc
$L_b'$ is transversely isotopic to $L_b$, by an isotopy fixed outside
the neighbourhood of $p$ and always intersecting $\PP$
transversely at the only point $p$, and $L_b'$ is good near the
intersection point $p$. 

After applying this {\em wrinkling} operation to all intersections
with $\PP$, we may hence assume that $L$ avoids a closed neighbourhood
$U_\PP$ of $S(1)\cup\dots\cup S(5)$.  

{\bf Step 4. }
After Step 3 and a perturbation of $L$, the curve $L_b$ intersects
the 2-dimensional submanifold $S(6)$ transversely in finitely many
points in $M\setminus U_\PP$. Thus we can repeat steps 2 and 3 with
$S(5)$ replaced by $S(6)$ to make $L$ disjoint from an open
neighbourhood $U_S$ of $S$.  

{\bf Step 5. } 
Now we are in the position to use the flow $\Psi_t$ of $X_t$ to push
$L_b$ into $V$. According to Lemma~\ref{lem:escape}
there exists $\tau>0$ such that $\Psi_\tau(M\setminus
U_S)\subset V$ and thus $\Psi_\tau(L_b)\subset V$. Since $\Psi_t$
preserves the binding and pages of the open book $(B,\pi)$, the
homotopy $(\xi_t,L_t)_{t\in[0,\tau]}$ defined by
$$
   \xi_t:=(\Psi_t)_*\xi,\qquad L_t:=\Psi_t(L)
$$
satisfies properties (a) and (b) above. Moreover, the bad
(unbraided) part of $\LL_\tau$ is given by $\Psi_\tau(L_b)$ and hence
contained in $V$. Note that the new contact form $\alpha_\tau$ is
given on $V$ by  
$$
  \alpha_\tau=T(1-r^2)\Bigl(d\phi+(1+\tau)r^2d\theta\Bigr).
$$
{\bf Step 6. }
Consider the contact structure $\xi_\tau=\ker\alpha_\tau$ from Step 5
on $V=S^1\times D^2$. After rescaling in $r$ we may assume that 
$\xi_\tau=\xi_{st}=\ker(d\phi+r^2d\theta)$. By a result of
Bennequin (\cite{Ben}, see also~\cite{Pav} for a short exposition) we
can transversely isotope $L_\tau$ relative to $\p V$ (fixing
$\xi_\tau$) to a link which is disjoint from the binding
$\{r=0\}$ and positively transverse to the pages $\{\theta=const\}$ in
$V$. Since $L_\tau$ was already braided outside $V$, the resulting
link $L_A\subset M$ is braided around $(B,\pi)$. 

Finally, let $K$ be the maximum of the intersection 
numbers of the components of $L_A$ with a page of $(B,\pi)$. Then for
any integer $k_i\geq K$ we can further transversely isotope the $i$-th
component (pulling it through $B$) in such a way that at the end it is
again braided and its intersection number with a page equals $k_i$.  
This concludes the proof of Theorem~\ref{thm:braiding}. 
\end{proof}

\section{Standardization of a contact open book near a 
transverse knot}\label{sec:standard} 

In this section we prove the following improvement of
Theorem~\ref{thm:braiding}. 

\begin{cor}\label{cor:braiding}
Let $(M,\alpha)$ be a closed oriented contact $3$-manifold.
Let $L=\cup_{i=1,\dots,\ell}L_i\subset M$ be
a link (with components $L_i$) positively transverse to the contact
structure $\ker\alpha$. Then there exists a natural number $K$ with the
following property. For every collection of numbers $k_i\ge K$,
$i=1,\dots,\ell$, there exists a contact form $\beta$ with
$\ker\beta=\ker\alpha$ and an open book $(B,\pi)$ supporting $\beta$
such that $L$ is disjoint from $B$ and positively 
transverse to the pages, and the intersection number of $L_i$ with a
page is $k_i$. Moreover, in coordinates $(\theta,r,\phi)\in S^1\times
D^2$ near $L_i=S^1\times\{0\}$ we have 
$$
   \beta=\alpha = f(d\theta+r^2d\phi)\quad \text{and}\quad
   \pi(\theta,r,\phi)= k_i\theta. 
$$
\end{cor}

\begin{proof}
By Theorem~\ref{thm:braiding} we find a contact form $\beta$ and an
open book $(B,\pi)$ having all the properties in the corollary except
the last one, i.e.~$\beta$ need not agree with $\alpha$ near
$L$ and $\pi$ need not be standard as above. 
Consider a component $L_i$ of $L$. It has a tubular neighbourhood
$V\cong S^1\times D^2_\delta$ with coordinates $\theta\in S^1$ and polar
coordinates $(r,\phi)$ on the disk $D^2_\delta$ in which the contact
structure $\xi=\ker\beta$ is given by 
$$
   \xi_\st = \ker\alpha_\st, \qquad  
   \alpha_\st = d\theta+r^2d\phi.
$$
By Lemma~\ref{lem:knot-openbook-structure} below there exist 
\begin{itemize}
\item an open book $(B,\pi')$ on $M$ which agrees with $(B,\pi)$
  outside $V$ such that $\pi'(\theta,r,\phi)=k_i\theta$ near $L_i$, and 
\item a contact form $\beta'$, supported by $(B,\pi')$ and defining
  $\xi$, which coincides with $\beta$ outside $V$ and with
  $\alpha_\st$ near $L_i$. 
\end{itemize}
By Lemma~\ref{lem:knot-openbook-form} below, $\beta'$ can be modified 
near $L_i$ to a contact form $\beta''$, still supported by $(B,\pi')$ and
defining $\xi$, such that $\beta''=\alpha$ near $L_i$. Then $\beta''$
and $(B,\pi')$ have the desired properties.
\end{proof}

It remains to prove the two lemmas used in the proof of
Corollary~\ref{cor:braiding}. 
We consider a tubular neighbourhood $V=S^1\times D^2_\delta$ of the
knot $K=\{r=0\}$ with coordinates $(\theta,r,\phi)$ and the contact
structure $\xi_\st=\ker\alpha_\st$ as above. Note that an open book
without binding on $V$ is simply a submersion $\pi:V\to
S^1$, and it supports a contact form $\alpha$ iff $d\pi\wedge
d\alpha>0$. 

\begin{lemma}\label{lem:knot-openbook-structure}
Let $\alpha$ be a contact form on $V$ defining $\xi_\st$ and
$\pi:V\to S^1$ be a submersion such that $d\pi\wedge
d\alpha>0$ and $\pi|_K:K\to S^1$ is a covering of degree
$n\in\N$. Then there exist 
\begin{itemize}
\item a submersion $\pi':V\to S^1$ which agrees with $\pi$ 
  near $\p V$ such that $\pi'(\theta,r,\phi)=n\theta$ near $K$, and  
\item a contact form $\alpha'$ on $V$, defining $\xi_\st$ and
  satisfying $d\pi'\wedge d\alpha'>0$, which coincides with $\alpha$
  near $\p V$ and with $\alpha_\st$ near $K$.
\end{itemize}
\end{lemma}

Note that for the submersion $\pi'(\theta,r,\phi)=n\theta$ the
condition $d\pi'\wedge d\alpha'>0$ is just $d\theta\wedge
d\alpha'>0$. 
 
\begin{lemma}\label{lem:knot-openbook-form}
Let $\alpha_0,\alpha_1$ be two contact forms on $V$ defining $\xi_\st$  
such that $d\theta\wedge d\alpha_i>0$ for $i=0,1$.
Then there exists a contact form $\alpha$ on $V$, defining $\xi_\st$
and satisfying $d\theta\wedge d\alpha>0$, which coincides with
$\alpha_0$ near $\p V$ and with $\alpha_1$ near $K$.
\end{lemma}

\begin{proof}[Proof of Lemma~\ref{lem:knot-openbook-form}]
We write $\alpha_i=f_i\alpha_\st$ for functions $f_i:V\to\R_+$ and
$\alpha=f\alpha_\st$, where
$$
   f := \bigl(1-\rho(r)\bigr)f_0 + \rho(r)f_1.
$$
Here $\rho:[0,\delta]\to[0,1]$ is a nonincreasing function as in
Lemma~\ref{lm:cutoff} which equals
$1$ near $0$ and $0$ near $\delta$ and satisfies
$|r\rho'(r)|\leq\eps$, for arbitrarily small constants
$\eps,\delta>0$ that will be chosen below. Now
$$
   d\theta\wedge d\alpha = d\theta\wedge\frac{\p f}{\p r}dr\wedge
   r^2d\phi + f\,d\theta\wedge 2r\,dr\wedge d\phi = \left(2f +
     r\frac{\p f}{\p r}\right) d\theta\wedge r\,dr\wedge d\phi
$$
is positive iff 
\begin{equation}\label{eq:f-pos}
   2f + r\frac{\p f}{\p r}>0.
\end{equation} 
To show~\eqref{eq:f-pos} we estimate with positive constants
$c_0,c_1,c_2$ depending only on $f_0,f_1$:
\begin{gather*}
   2f \geq 2\min\{\min_Vf_0,\min_Vf_1\} \geq c_0 > 0, \cr 
   \left|r\frac{\p f}{\p r}\right| 
   \leq r\left|(1-\rho)\frac{\p f_0}{\p r} + \rho\frac{\p f}{\p
       r}\right| + |r\rho'(r)|\,|f_1-f_0| 
   \leq c_1\delta + c_2\eps.
\end{gather*}
Thus for $\eps,\delta$ sufficiently small~\eqref{eq:f-pos} holds and
Lemma~\ref{lem:knot-openbook-form} follows. 
\end{proof}

\begin{proof}[Proof of Lemma~\ref{lem:knot-openbook-structure}]
{\bf Step 1. }
We make a coordinate change of $V$ of the form
$$
   (\theta,r,\phi)\mapsto \bigl(\Theta=\Theta(\theta),r,\phi\bigr), 
$$
where the diffeomorphism $\theta\mapsto \Theta(\theta)$ of $K\cong
S^1$ is chosen in such a way that in the new coordinates the covering
on $K$ is given by
\begin{equation}\label{eq:covering}
   \pi(\Theta,0,\phi)=n\Theta. 
\end{equation}
Note that $d\theta\wedge d\alpha$ is positive at $\{r=0\}$ by the
contact condition, so after shrinking $V$ we may assume it is
positive on the whole of $V$ and thus
$$
   d\Theta\wedge d\alpha = \Theta'(\theta)d\theta\wedge d\alpha > 0.    
$$
{\bf Step 2. }
Since the 1-forms $d\pi$ and $n\,d\Theta$ are cohomologous, we have
$$d\pi=n\,d\Theta+df$$ for some function 
$f:V\to\R$, hence (after adding a constant to $f$ if necessary)
$\pi=n\Theta+f$. Equation~\eqref{eq:covering} then implies that
$f|_K=0$, i.e.~we have $|f|\leq C_fr$ for some constant $C_f$
depending only on $f$. We take a nondecreasing function
$\rho:[0,\delta]\to[0,1]$ as provided by Lemma~\ref{lm:cutoff}
(replacing $\rho$ by $1-\rho$) which equals $0$ near $r=0$ and $1$
near $r=\delta$ and satisfies $r\rho'(r)\leq \eps$. Consider the map
$$
   \pi_1:=n\Theta+\rho(r)f:V\to S^1.
$$
{\bf Claim: }For $\eps$ sufficiently small we have
$d\pi_1\wedge d\alpha>0$. 

To prove this, we use $\alpha=h(d\theta+r^2d\phi)$ to write out 
\begin{align}\label{eq:main}
   d\theta_1\wedge d\alpha
   &= d(n\Theta+\rho f)\wedge d\alpha \cr
   &= (n\,d\Theta+\rho\,df)\wedge d\alpha + f\,d\rho\wedge
   d\Bigl(h(d\theta+r^2d\phi)\Bigr) \cr 
   &= (n\,d\Theta+\rho\,df)\wedge d\alpha +
   f\rho'(r)[h_\phi-r^2h_\theta]dr\wedge d\phi\wedge d\theta.
\end{align}
Since both expressions $$S_1:=n\,d\Theta\wedge d\alpha$$ and
$$S_2:=(nd\Theta+df)\wedge d\alpha=d\pi\wedge d\alpha$$
are positive, their convex combination 
$$(1-\rho)S_1+\rho S_2=
(n\,d\Theta+\rho\,df)\wedge d\alpha\geq min(S_1,S_2)$$
is bounded from below by a constant independent of $\rho$.
So it remains to estimate the last summand in~\eqref{eq:main}. 
Note that the $\phi$-derivative of any function vanishes at $\{r=0\}$
and thus we have an estimate $|h_\phi|\leq C_hr$ with a constant $C_h$
depending only on $h$. Using this as well as $|f|\leq C_fr$ and
$r\rho'(r)\leq \eps$ we estimate 
$$
   |f\rho'(r)[h_\phi-r^2h_\theta]|\leq C_fr\rho'(r)C_hr\leq C_fC_h\eps
   r. 
$$
Since $r\,dr\wedge d\phi\wedge d\theta$ is a smooth volume form, this
shows that the last summand in~\eqref{eq:main} becomes 
arbitrarily small for $\eps$ small and thus proves the claim. 

Thus $\pi_1:V\to S^1$ is a submersion, satisfying $d\pi_1\wedge
d\alpha>0$, which agrees with $\pi$ near $\p V$ and with $n\Theta$
near $K$. After shrinking $V$ and renaming $\pi_1$ back to $\pi$ we
may hence assume that $\pi=n\Theta$ on $V$.  

{\bf Step 3. }
Note that, after Step 2, $\alpha$ and $\alpha_\st$ are two contact
forms defining $\xi_\st$ with $d\Theta\wedge\alpha>0$ and
$d\Theta\wedge\alpha_\st>0$. So by Lemma~\ref{lem:knot-openbook-form}
(in coordinates $(\Theta,r,\phi)$) we find a contact form $\beta$ on
$V$, defining $\xi_\st$ and satisfying $d\theta\wedge d\beta>0$, which
coincides with $\alpha$ near $\p V$ and with $\alpha_\st$ near $K$.
After shrinking $V$ and renaming $\beta$ back to $\alpha$, we may
hence assume that $\alpha=\alpha_\st$ and $\pi=n\Theta$ on $V$. 

{\bf Step 4. }
It remains to modify $\pi$ such that it equals $n\theta$ near $K$. 
The argument is similar to Step 2 but simpler: The forms $d\Theta$ and
$d\theta$ are cohomologous, so we have
$n\,d\Theta=n\,d\theta+df$ for some function $f$. In other words
(after adding a constant to $f$ if necessary) $n\Theta=n\theta+f$. We
define 
$$
   \pi':=n\theta+\rho(r) f:V\to S^1
$$ 
with a cutoff function $\rho$ as in Step 2. Since
$d\alpha_\st=2r\,dr\wedge d\phi$ we have\\ $f\,d\rho\wedge
d\alpha_\st=0$ and thus
$$
   d\pi'\wedge d\alpha_\st = (n\,d\theta + \rho\,df)\wedge
   d\alpha_\st,
$$
which is a convex combination of the positive terms $n\,d\theta\wedge
d\alpha_\st$ and $n\,d\Theta\wedge d\alpha_\st$ and hence positive. 
This concludes the proof of Lemma~\ref{lem:knot-openbook-structure}. 
\end{proof}

\section{Contact open books with boundary}\label{sec:relGiroux}

Consider an oriented contact 3-manifold $(N,\alpha)$ whose boundary $\p
N=T_1\amalg\dots\amalg T_\ell$ is a union of 2-tori. We assume that
$\alpha$ is $T^2$-invariant near each boundary component $T_i$,
i.e.~there exists a collar neighbourhood $K_i\cong [l_i,R_i]\times
T^2$ of $T_i=\{r=l_i\}\cong T^2$ with oriented coordinates
$(r,\phi,\theta)$ in which $\alpha$ is given by
\begin{equation}\label{eq:alpha_h}
   \alpha_{h}=h_{1}(r)d\phi+h_{2}(r)d\theta
\end{equation}
for some immersion $h:[l_i,R_i]\to\C$ satisfying 
\begin{equation}\label{eq:h}
   h_1'h_2-h_2'h_1 > 0.  
\end{equation}
(The immersion $h$ of course depends on $i$, but we suppress this
dependence to keep the notation simple). 
We are interested in $\alpha$ that are supported by an open book
decomposition $(B,\pi)$ in the usual sense (i.e.~$\alpha>0$ on the
binding $B$ and $d\alpha>0$ on the interior of the pages), where
$B\subset\inn N$ and the projection $\pi:N\setminus B \to S^1$ is in
the above coordinates near each $T_i$ a linear projection
$\pi(r,\phi,\theta)=a_{1}\phi+a_{2}\theta$. Here $a\in\Z^2$ is
some integer vector (again depending on $i$) and positivity of
$d\alpha$ on the pages, $d\pi\wedge d\alpha>0$, is equivalent to 
$$
   h_1'a_2-h_2'a_1 > 0.  
$$
(Such ``relative contact open books'' were previously considered
in~\cite{VHM,BEV}).  
After a linear coordinate change on $T_i$ we may then assume that
$a=(0,n_i)$ for some $n_i\in\N$, so $\pi(r,\phi,\theta)=n_i\theta$ and
the positivity condition $d\theta\wedge d\alpha>0$ becomes
\begin{equation}\label{eq:h1}
   h_1'>0. 
\end{equation}
Note that the page $\Sigma$ of the open book has two kinds of boundary
components: those that get collapsed to binding components, and others
that give rise to the boundary tori $T_i$. 

Recall Giroux's result~\cite{Gir} that for any cooriented contact structure on a
closed oriented 3-manifold (without boundary) there exists a defining
contact form which is supported by an open book. The goal of this
section is to prove the following relative version. 

\begin{prop}\label{prop:relGiroux}
Let $(N,\alpha)$ be a compact oriented contact manifold with toric
boundary $\p N=T_1\amalg\dots\amalg T_\ell$ as above, i.e.~there exist
collar neighbourhood $K_i\cong [l_i,R_i]\times T^2$ 
of $T_i=\{r=l_i\}\cong T^2$ on which $\alpha$ is given
by~\eqref{eq:alpha_h}. Assume that on each $K_i$ we have 
$d\theta\wedge d\alpha>0$ (i.e.~$h_1'>0$). 

Then there exists a natural number $K$ such that for any collection
$\{n_i\}_{i=1,\dots,\ell}$ of natural numbers with $n_i\ge K$ 
there exist a homotopy of contact form $\alpha_t$, $t\in[0,1]$, and an
open book decomposition $(B,\pi)$ of $N$ with the following properties:
\begin{enumerate}
\item The contact form $\alpha_1$ is supported by the open book
  decomposition $(B,\pi)$. Moreover,
  $d\alpha_1=d\alpha$ near $\p N$. 
\item Near $T_i$ the fibration $\pi$ is given by
  $\pi(\theta,r,\phi)=n_i\theta$.
\item $\alpha_0=\alpha$, and $\alpha_t$
  is $T^2$-invariant near $\p N$ for all $t$. 
\item Writing $\alpha=h_1(r)d\phi+h_2(r)d\theta$ and
  $\alpha^t=h_1^t(r)d\phi+h_2^t(r)d\theta$ near $T_i$, we have
  $h_1^1(l_i)>0$, and if $h_1(l_i)>0$ then $h_1^1(l_i)<h_1(l_i)$.  
Moreover, if $h_1(l_i)\le 0$, then $(h^t)'(l_i)$ makes one full {\em
  negative} (i.e.~clockwise) turn in the $(h_2,h_1)$ plane as $t$ runs
from $0$ to $1$. If $h_1(l_i)>0$, then we have a choice: we can choose
$(h^t)'(l_i)$ to have rotation number $-1$ or $0$. 
\item If on some $T_i$ we have $h(r)=((r-l_i),1)$, then instead of
  (iv) the $T^2$-invariant homotopy $\alpha_t$ near $T_i$ can be
  chosen to be constant. 
\end{enumerate}
\end{prop}

\begin{proof}
{\bf Special case. }
First we treat the special case that $h(r)=((r-l_i),1)$ near each
$T_i$. After a shift in the $r$ coordinate
by $l_i$ we may assume that $\alpha=d\theta+rd\phi$ and $r\in
[0,R_i]$. We remove the torus $\{0\}\times T^2$ and on the remaining
$(0,R_i]\times T^2$ we do a coordinate change $r_1:=r^{1/2}$ and then
rename $r_1$ back to $r$. Now the $1$-form $\alpha$ 
looks like $d\theta+r^2d\phi$. So it extends to the solid torus
obtained by collapsing the $\phi$-direction in the removed torus
$\{0\}\times T^2$. This gives rise to a closed contact manifold $(\bar
N,\bar\alpha)$ with a transverse link $K_1\amalg\dots\amalg K_\ell$
obtained from the $T_i$. We apply Corollary~\ref{cor:braiding} 
to obtain a contact form $\bar\beta$ on $\bar N$,
defining the same contact structure as $\bar\alpha$ and supported by an
open book $(\bar B,\bar\pi)$, such that $\bar\beta=\bar\alpha$ and
$\bar\pi=n_i\theta$ near $K_i$. Replacing $K_i$ back by $T_i$ and
changing coordinates back from $r^{1/2}$ to $r$ yields a contact
homotopy $\alpha_t$ (by linear interpolation from $\bar\alpha$ to
$\bar\beta$) and open book satisfying conditions (i-iii) and (v) in
the proposition.   
This was ``dream situation'' in which we did not have to worry how to
get back from $\bar N$ to $N$. Now we turn to the 
\smallskip
 
{\bf General case. }
Consider one neighbourhood $K_i=[l_i,R_i]\times T^2$. After a
shift we may assume that $l_i>0$. 
We consider the solid torus $S^1\times D^2$ with coordinates
$\theta\in S^1$ and polar coordinates $(r,\phi)$ on $D^2$, where $r\in
[0,l_i]$, so that we can write $S^1\times D^2=\{r\in [0,l_i]\}$.
We identify the boundary of this solid torus with $T_i$ via the
identity map. Gluing in these solid tori for all $i$ gives us a closed
manifold $\bar N$. Let $L_i$ denote the core circle $\{r=0\}$ of the
solid torus $\{r\in [0,l_i]\}$. We consider the union of this solid torus
with the collar neighbourhood of the respective boundary component
$$
   V_i:=K_i\cup \{r\in [0,l_i]\}=\{r\in [0,R_i]\}.
$$ 
We extend the contact form $\alpha$ from $N$ to $\bar N$ as
follows. Recall that on $K_i$ we have $\alpha=\alpha_h$ for an
immersion $h:[l_i,R_i]\to\C$ satisfying~\eqref{eq:h}
and~\eqref{eq:h1}. We extend $h$ from $[l_i,R_i]$
to $[0,R_i]$ so that the contact condition~\eqref{eq:h}
holds and $h(r)=(r^2,1)$ near $r=0$. Moreover, if $h_1(l_i)>0$ we
arrange that $h_1'>0$ on $(0,l_i]$ (see Figure \ref{figure:h1}), and if
$h_1(l_i)\leq 0$ we let $h$ rotate as in Figure \ref{figure:h2}.

\begin{figure}
\centering
\includegraphics[height=8cm]{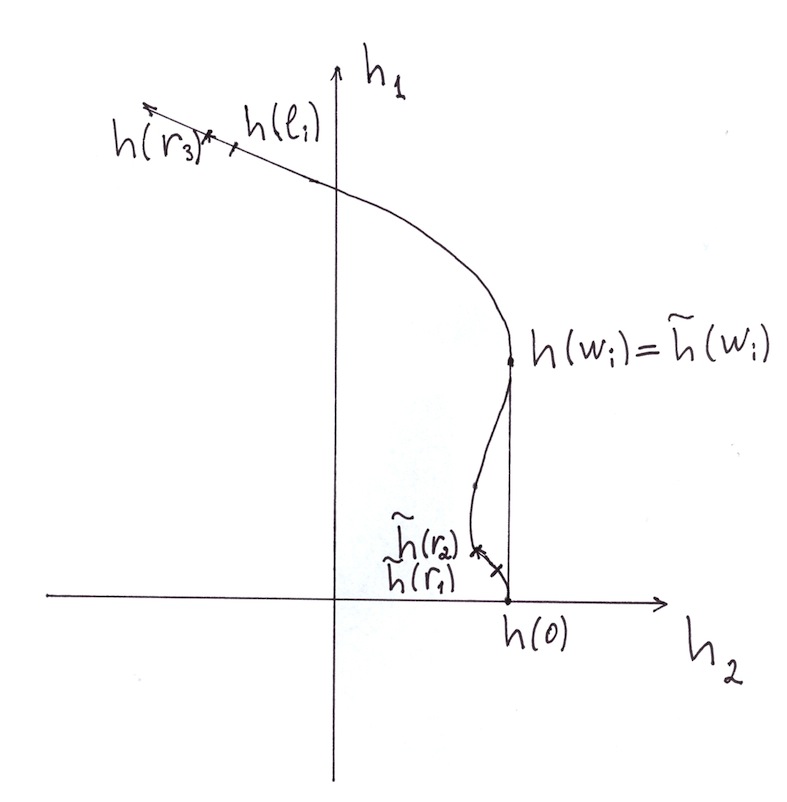}
\caption{The case $h_1(l_i)>0$}
\label{figure:h1}
\end{figure}

\begin{figure}
\centering
\includegraphics{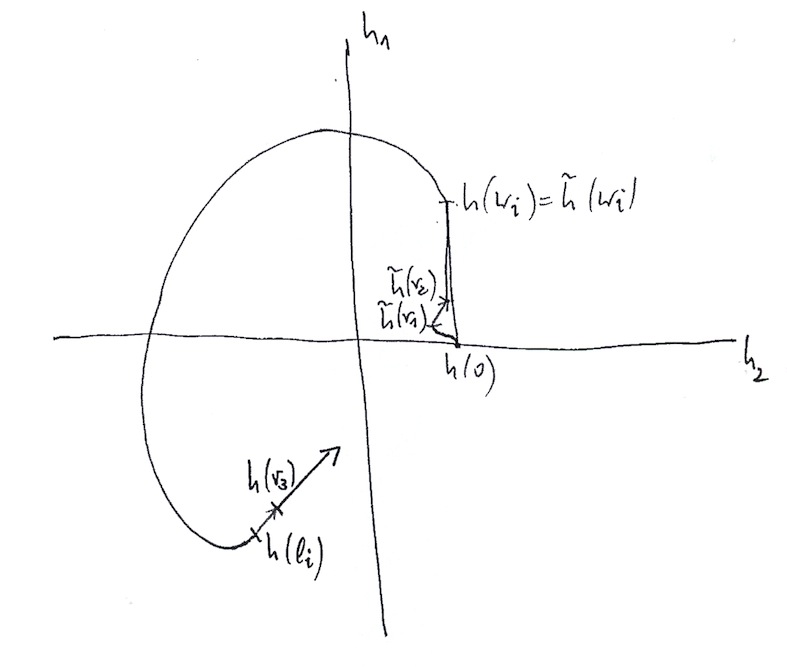}
\caption{The case $h_1(l_i)\leq 0$}
\label{figure:h2}
\end{figure}

The condition $h(r)=(r^2,1)$ near $r=0$ means that near $L_i$ we have 
$\alpha=\alpha_\st$ with 
$$
  \alpha_\st:=d\theta+r^2d\phi,
$$ 
which extends smoothly over $r=0$. 
This gives us the desired extension of $\alpha_h$ from $K_i$ to $V_i$
and thus the extension of $\alpha$ from $N$ to $\bar N$. To get back
to $N$ we just have to cut out the solid tori $\{r<l_i\}$.  

Let $\xi:=\ker\alpha$ denote the contact structure defined by $\alpha$
on $\bar N$. An application of Corollary~\ref{cor:braiding} to $(\bar
N,L:=\cup L_i,n_i,\alpha)$ gives us a new defining form $\beta$ for 
$\xi$ and an open book $(\bar B,\bar\pi)$ with the following properties. The
contact form $\beta$ is supported by the open book 
$(\bar B,\bar\pi)$, and both the open book projection 
and the contact form restrict to a neighbourhood 
$$
   W_i:=\{r\le w_i\}
$$ 
of $L_i$ in a standard way: $\bar\pi(r,\phi,\theta)=n_i\theta$ and
$\beta=\alpha=\alpha_\st$. 

Unfortunately, neither $\beta$ nor $\bar\pi$ is nice on $K_i$: It may
even happen that the binding $\bar B$ intersects $K_i$, and the form
$\beta$ need not be
$T^2$-invariant on $K_i$. In order to take care of this,
we introduce a new contact structure and new contact forms. We begin
by picking a subdivision 

\begin{figure}
\centering
\includegraphics{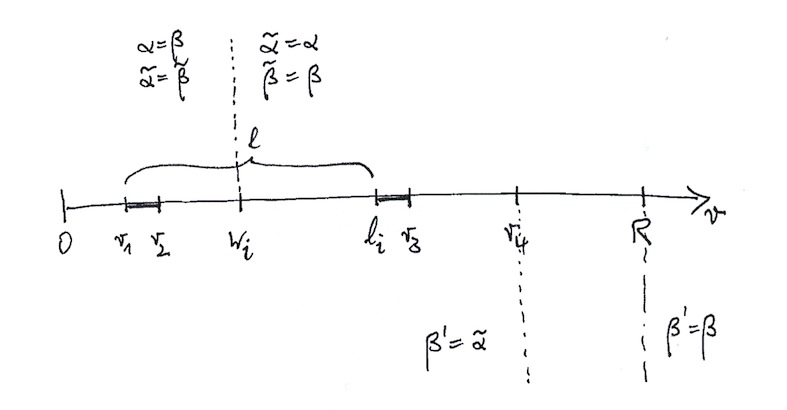}
\caption{The subdivision}
\label{figure:subdiv}
\end{figure}

\begin{equation}\label{eq:subdiv} 
   0<r_1<r_2<w_i<l_i<r_3<R_i.
\end{equation}
and an immersion $\tilde h:[0,w_i]\to\C$ with the following properties
(see Figures \ref{figure:h1}, \ref{figure:h2} and \ref{figure:subdiv}):
\begin{itemize}
\item $r_2-r_1=r_3-l_i$, or equivalently, $r_3-r_2=l_i-r_1=:l$; 
\item near $0$ and $w_i$ we have $\tilde h(r)=(r^2,1)$;
\item on $[r_1,r_2]$ we have $\tilde h(r)=h(r+l)+(A,B)$ for some
  constant $(A,B)\in \C$; 
\item the function $\tilde h$ satisfies conditions~\eqref{eq:h}
  and~\eqref{eq:h1}.
\end{itemize}
Of course $r_1,r_2,r_3,\tilde h$ depend on $i$, but we suppress this
from the notation. We extend $\tilde h$ as $h$ over the interval
$[w_i,R]$. This defines a contact form $\alpha_{\tilde h}$ on each $V_i$
and thus a contact form $\tilde\alpha$ on $\tilde N$ which differs
from $\alpha$ only on the neighbourhoods $W_i$. Recall that $\alpha$
and $\beta$ coincide on $W_i$, so we can define
a contact form $\tilde\beta$ on $\bar N$ as $\tilde\alpha$ on each
$W_i$ and as $\beta$ on the rest of $\bar N$. It is crucial to note
that the $\tilde h_1'>0$ implies that the contact form $\tilde\beta$
is supported by $(\bar B,\bar\pi)$.
Note also that the contact forms $\tilde\alpha$ and $\tilde\beta$
define the same contact structure that we denote by $\tilde\xi$. 

We introduce a third contact
form defining $\tilde\xi$. For this, we introduce another subdivision
point into subdivision~\eqref{eq:subdiv}, namely $r_4$: 
$$
   0<r_1<r_2<w_i<l_i<r_3<r_4<R.
$$
Now we pick a contact form $\beta'$ defining $\tilde\xi$ 
that coincides with $\tilde\alpha$ on $\{r\in [0,r_4]\}$ 
and with $\beta=\tilde\beta$ on a neighbourhood of $\bar N\setminus
(\cup V_i)$. 

We need one more ingredient: an isotopy of diffeomorphisms
$\{\Phi_t\}_{t\in [0,1]}$ of the interval $[0,R]$ with the following
properties (see Figure \ref{figure:Phi}):
\begin{itemize}
\item $\Phi_0=\id$;
\item $\Phi_t=\id$ near $0$ and $r_4$ for all $t\in [0,1]$; 
\item for $r\in [l_i,r_3]$ we have $\Phi_1(r)=r-l$.
\end{itemize}
\begin{figure}
\centering
\includegraphics{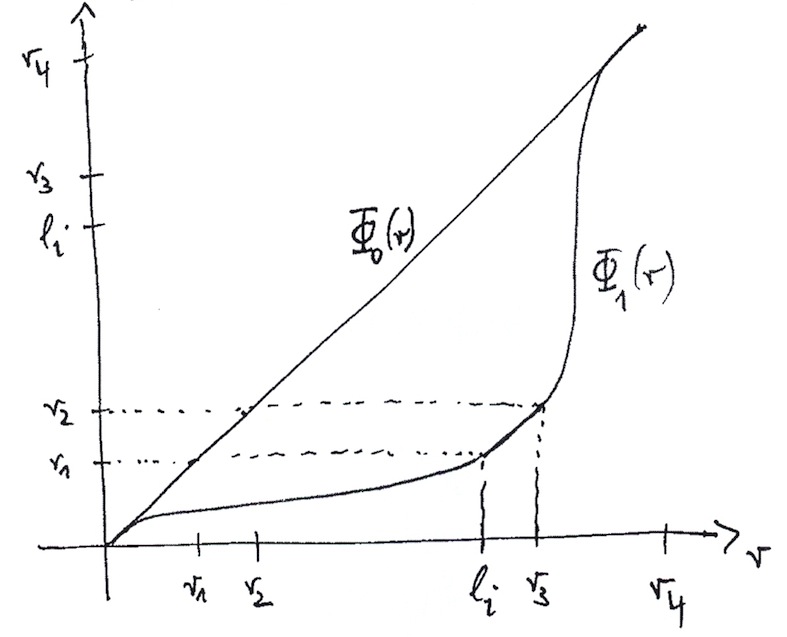}
\caption{The diffeomorphism $\Phi$}
\label{figure:Phi}
\end{figure}
For each $\Phi_t$ we also denote the induced radial diffeomorphism of
$V_i$ by $\Phi_t$. 
This diffeomorphism extends from $\cup V_i$ to the rest of $\bar N$ by
setting it to be the identity on the complement of $\cup V_i$. The
resulting diffeomorphism will be denoted as $\Phi_t$ again. This
finishes the necessary constructions and we now describe a homotopy of
contact forms on $\bar N$ in 3 steps. The desired homotopy $\alpha_t$
will be obtained by restricting this homotopy to $N$. 

\smallskip
 
{\bf Step 1. } 
We start with the contact form $\tilde\alpha$ and linearly homotop it
to $\beta'$ 
(they define the same contact structure $\tilde\xi$). Since $\beta'$
and $\tilde\alpha$ coincide on $\{r\le r_4\}$, the homotopy is fixed
there, in particular on $\{r\in [l_i,r_3]\}$. 

{\bf Step 2. } 
We deform $\beta'$ by setting $\beta_t':=\Phi_t^*\beta'$. On $\{r\in
[0,r_4]\}$ we have
$\beta_t'=\Phi_t^*\tilde\alpha=\Phi_t^*\alpha_{\tilde
  h}=\alpha_{\tilde h\circ\Phi_t}$. Since $\Phi_t=\id$ near $0$ and
$r_4$, the last equality is justified and we see that $\beta_t'$ is
$T^2$-invariant on $\{r\in [0,r_4]\}$ (so in particular on 
$\{r\in [l_i,r_3]\}$) for all $t$. 

{\bf Step 3. } 
Finally, we deform $\Phi_1^*\beta'$ by deforming $\beta'$ to
$\tilde\beta$. Namely, we consider the linear homotopy
$(1-t)\beta'+t\tilde\beta$ between the two forms $\beta',\tilde\beta$
defining the contact structure $\tilde\xi$ and set  
$$
   \beta_t:=\Phi_1^*\bigl((1-t)\beta'+t\tilde\beta\bigr).
$$ 
Note that on $\{r\in [0,w_i]\}$ we have 
$$
   (1-t)\beta'+t\tilde\beta=\alpha_{\tilde h}
$$ and thus on $\{r\in [l_i,r_3]\}$ we see that 
$$
   \beta_t = \Phi_1^*\alpha_{\tilde h} = \alpha_{\tilde h\circ\Phi_1}
   = \alpha_{\tilde  h(r-l)} = \alpha_{h(r)+(A,B)} =  
   \alpha_h+A\,d\phi+B\,d\theta = \alpha +A\,d\phi+B\,d\theta
$$ 
is fixed. 

We define the homotopy $\alpha_t$ by restricting the homotopy
constructed in Steps 1-3 above to $N$, thus $\alpha_0=\alpha$ and
$\alpha_1=\Phi_1^*\tilde\beta|_N$. We define the 
open book decomposition
$(B,\pi):=(\Phi_1^{-1}(\bar B)|_N,\bar\pi\circ\Phi_1)$ on $N$. Now
we check conditions (i-v) in the proposition with ``$\{r\in
[l_i,r_3]\}$'' in place of ``near $\p N$''. 
\smallskip

(i) The contact form $\Phi_1^*\tilde\beta$ is supported by the open book
$(\Phi_1^{-1}(\bar B),\bar\pi\circ\Phi_1)$ because $\tilde\beta$ is
supported by $(\bar B,\bar\pi)$. The last computation in Step 3 above
shows that on $\{r\in[l_i,r_3]\}$ we have
$d\alpha_1=d\alpha_h=d\alpha$. 
 
For (ii) note that $\bar\pi(\theta,r,\phi)=n_i\theta$ on $\{r\in
[r_1,r_2]\}$ implies that
$\bar\pi\circ\Phi_1(\theta,r,\phi)=n_i\theta$ on $\{r\in [l_i,r_3]\}$.

(iii) In the construction in Steps 1-3 above we always
checked the $T^2$-invariance of the contact forms on $\{r\in
[l_i,r_3]\}$. 

(iv) Recall that the homotopies in Steps 1 and 3 are constant near
$T_i=\{r=l_i\}$, so the homotopy $\alpha_t$ is given near $T_i$ by the
homotopy $\beta_t'=\alpha_{h^t}$ in Step 2, where
$h^t:=\tilde h\circ\Phi_t$. Thus for $t=1$ the first component
satisfies $h^1_1(l_i)=\tilde h_1(r_1)>0$, and if $h_1(l_i)>0$ then 
$h^1_1(l_i)=\tilde h_1(r_1)<\tilde h_1(l_i)=h_1(l_i)$ because we chose
$\tilde h_1$ strictly increasing on $(0,l_i]$ in this
case. Since $\Phi_t(l_i)$ decreases from $l_i$ to $r_1$ as
$t$ increases from $0$ to $1$, Figures \ref{figure:h1} and \ref{figure:h2}
show that the derivatives $(h^t)'(l_i) = \Phi_t'(l_i)\tilde
h'\bigl(\Phi_t(l_i)\bigr)$ have rotation number $0$ if $h_1(l_i)>0$
and $-1$ if $h_1(l_i)<0$. Moreover, by adding positive rotations in
the construction of $h$, we can decrease the rotation numbers by
arbitrary integers, in particular we can make it $-1$ in the case
$h_1(l_i)>0$. 

(v) can be arranged as in the special case discussed at the beginning
of this proof.  
\end{proof}

\section{The structure theorem}\label{sec:structure}

In this section we collect and refine some results from~\cite{CV} that
will be needed in the proof in Section~\ref{sec:proof}. Most
importantly, we will use the following structure theorem. 

\begin{theorem}\label{thm:structure}
Let $(\om,\lambda)$ be a SHS on a closed $3$-manifold $M$ and set
$f:=d\lambda/\om$. 
Then there exists a (possibly disconnected and possibly with boundary)
compact $3$-dimensional submanifold $N$ of $M$, invariant under the Reeb flow,
a (possibly empty) disjoint union
$U=U_1\cup\dots\cup U_k$ of compact integrable regions
and a stabilizing 1-form $\tilde\lambda$ for $\om$ with the following
properties: 
\begin{itemize}
\item $\inn U\cup \inn N=M$;
\item the proportionality coefficient 
$\tilde f:=d\tilde\lambda/\om$ is constant on each connected component of $N$; 
\item on each $U_i\cong [0,1]\times T^2$ the SHS $(\om,\lambda)$ is
  $T^2$-invariant and $f(r,z)=\alpha_ir+\beta_i$ for constants
  $\alpha_i>0$, $\beta_i\in\R$;
\item $\tilde\lambda$ is $C^1$-close to $\lambda$. 
\end{itemize}
Moreover, if $\om$ is exact we can arrange that $\tilde f$ attains
only nonzero values on $N$. 
\end{theorem}

We will also use the following result from~\cite{CV}. 

\begin{prop}\label{prop:thick}
Let $(\om,\lambda)$ be a SHS on a closed $3$-manifold $M$ and set
$f:=d\lambda/\om$. Let $Z\subset\R$ be any set of Lebesgue measure
zero containing a value $a\in Z\cap\im(f)$. 
Then there exists a stabilizing form $\tilde\lambda$ for $\om$ such
that $\tilde f:=d\tilde\lambda/\om$ can be written as 
$\tilde f=\sigma\circ f$ for a function $\sigma:\R\to\R$ 
which is locally constant on a open
neighbourhood of $Z$ (and thus $\tilde f$ is locally constant on an
open neighbourhood of  $f^{-1}(Z)$) and $\tilde f\equiv a$ on
$f^{-1}(a)$.  
Moreover, for every $s\in[0,1)$ we can achieve that $\tilde\lambda$ is
$C^{1+s}$-close to $\lambda$. 
\end{prop}

Using this, we now prove the following refinement of
Theorem~\ref{thm:structure}. 

\begin{corollary}\label{cor:structure2}
Every stable Hamiltonian structure on a closed 3-manifold $M$ is 
stably homotopic to a SHS $(\om,\lambda)$ for which 
there exists a (possibly disconnected and possibly with boundary)
compact $3$-dimensional submanifold $N=N^+\cup N^-\cup N^0$ of $M$,
invariant under the Reeb flow, and a (possibly empty) disjoint union
$U=U_1\cup\dots\cup U_k$ of compact integrable regions with the
following properties:  
\begin{itemize}
\item $\inn U\cup \inn N=M$;
\item the proportionality coefficient $f:=d\lambda/\om$ is constant
  positive resp.~negative on each connected component of $N^+$
  resp.~$N^-$;  
\item on each $U_i\cong [0,1]\times T^2$ the SHS $(\om,\lambda)$ is
  $T^2$-invariant and $f$ is nowhere zero;
\item on $N^0$ there exists a closed $1$-form $\bar\lambda$
  representing a primitive integer cohomology class $\bar\lambda\in
  H^1(N^0;\Z)$ such that $\bar\lambda\wedge\om>0$ and $\bar\lambda$ is
  $T^2$-invariant near $\p N^0$.
\end{itemize}
\end{corollary}

\begin{proof}
Consider a SHS $(\om,\lambda)$ and apply the Structure
Theorem~\ref{thm:structure} to find a new stabilizing 1-form
$\tilde\lambda$. Denote by $N^+,N^0,N^-$ the union of
components of $N$ on which $\tilde f=d\tilde\lambda/\om$ is positive
(resp.~zero, negative). If the original proportionality coefficient
$f=d\lambda/\om$ is nowhere zero, then we may assume that the new
proportionality coefficient $\tilde f$ is nowhere zero and
$(\om,\tilde\lambda)$ has all the desired properties (with
$N^0=\emptyset$). So suppose that $f$ has nonempty zero set
$f^{-1}(0)$. The proof of Theorem~\ref{thm:structure} (using
Proposition~\ref{prop:thick} with $Z$ the set of critical values
together with the value $a=0$) allows us to arrange that $f^{-1}(0)$
is contained in the interior of the flat part $N^0$. 

The new stabilizing form 
$\tilde\lambda$ restricts as a closed form to $N^0$. We $C^1$-perturb 
$\tilde\lambda|_{N^0}$ to get a $1$-form $\hat\lambda$ on
$N^0$ representing a rational cohomology class. 
Let $V\cong [0,1]\times T^2\subset N^0$ be a part of an integrable
region sitting in $N^0$ as a collar neighbourhood of one of its
boundary components $\{1\}\times T^2$. Since the restriction 
$\tilde\lambda|_V$ is $T^2$-invariant and $\hat\lambda$ is $C^1$-close
to $\tilde\lambda$, the $T^2$ average $\lambda_{inv}$ of $\hat\lambda$
on $V$ is $C^1$-close to $\hat\lambda$. As $\lambda_{inv}$ and
$\hat\lambda$ represent the same cohomology class on $V$, we can write
$\hat\lambda=\lambda_{inv}+d\chi$ for a smooth function 
$\chi$ on $V$. Moreover, $C^1$-closeness of $\lambda_{inv}$ and
$\hat\lambda$ allows us to choose $\chi$ also $C^1$-small. Let $\rho$
be a cutoff function on $[0,1]$ which equals $1$ near $0$ and $0$ near
$1$. Set $\bar\lambda:=\lambda_{inv}+d(\rho\chi)$ on $V$ and extend this form
as $\hat\lambda$ inside $N^0$. 
The closed form $\bar\lambda$ is $T^2$-invariant near the boundary of
$N^0$. Note also that 
$$
   \bar\lambda-\hat\lambda = \lambda_{inv}+d(\rho\chi)-
   (\lambda_{inv}+d\chi) = d(\rho\chi)-d\chi = d(\chi(\rho-1))
$$
on $V$. Therefore, the difference $\bar\lambda-\hat\lambda$ is exact
on $N^0$ and thus $\bar\lambda$ represents a rational cohomology
class. Assume the procedure above has been performed near all boundary
components of $N^0$. Now $C^1$-smallness of $\chi$ implies
$C^1$-smallness of $\rho\chi$. This together with the computation
above and the $C^1$-smallness of the difference
$\hat\lambda-\tilde\lambda$ ensures that $\bar\lambda\wedge\om>0$ on
$N^0$. 
After multiplying $\bar\lambda$ with a rational
number we may assume that it represents a primitive integer cohomology
class in $H^1(N^0,\Z)$. 

The set $N^0$ and the $1$-form $\bar\lambda$ on it have the
desired properties. However, the new proportionality coefficient
$\tilde f$ may still vanish on some integrable region $U_i$. 
To remedy this, we choose $\delta>0$ so small that
$\{f<\delta\}\subset N^0$. Now we apply Theorem~\ref{thm:structure}
{\em again} to the original SHS $(\om,\lambda)$ to obtain new
$\tilde\lambda,\tilde f$ and new regions $\tilde N^\pm,\tilde
N^0,\tilde U_i$. Moreover, we can make 
$\|\tilde\lambda-\lambda\|_{C^1}$ so small that 
$\tilde N^0\subset\{\tilde f=0\}\subset\{f<\delta\}\subset N^0$.  
In particular, $\tilde f$ does not vanish on a neighbourhod $W$ of
$M\setminus \inn N^0$. Now the SHS $(\om,\tilde\lambda)$, the {\em
  old} set $N^0$ and 1-form $\bar\lambda$, and the intersections of
the {\em new} sets $\tilde U_i$ and $\tilde N^\pm$ with $W$ satisfy
all conditions in Corollary~\ref{cor:structure2}.  
\end{proof}

\begin{remark}\label{rem:positive}
If $f$ is non-negative, then so is the new proportionality coefficient
constructed in the proof of Corollary~\ref{cor:structure2} and hence
$N^-=\emptyset$. Then all binding components of the open book
constructed in the proof of Theorem~\ref{thm:Giroux} in the next
section occur in $N^+$ and thus have positive signs. 
\end{remark}

\section{Proof of the main theorem}\label{sec:proof}

In this section we prove Theorem~\ref{thm:Giroux}. 

We will need two more technical results. As before, we consider
$[a,b]\times T^2$ with cordinates $(r,\phi,\theta)$. 
For functions $h=(h_1,h_2):[a,b]\to \C$ and $g=(g_1,g_2):[a,b]\to \C$
we define 1-forms
\begin{equation*}
   \alpha_{h}=h_{1}(r)d\phi+h_{2}(r)d\theta,\qquad
   \lambda_{g}=g_{1}(r)d\phi+g_{2}(r)d\theta. 
\end{equation*}
Then $(d\alpha_h,\lambda_g)$ is a SHS iff
\begin{equation}\label{eq:g}
   \la g',ih'\ra=0,\qquad \la g,ih'\ra>0.
\end{equation}
The following result is proved in~\cite{CV}. 

\begin{proposition}\label{prop:stabhom}
Let $h_t:[0,1]\to \C$, $t\in [a,b]$ be a 
homotopy of immersions such that for each $t$ the slope $h_t'/|h_t'|$
restricted to $[\eps,1-\eps]$ is nonconstant. Let $\bar
g_t:[0,\eps]\cup [1-\eps,1]$, $t\in [a,b]$ be a homotopy such that
$(h_t,\bar g_t)$ satisfies~\eqref{eq:g} on $[0,\eps]\cup [1-\eps,1]$
for all $t\in [a,b]$. Then there exists a homotopy $g_t$ which agrees
with $\bar g_t$ to $[0,\eps]\cup [1-\eps,1]$ such that $(h_t,g_t)$
satisfies~\eqref{eq:g} on $[0,1]$ for all $t\in [a,b]$.
\end{proposition}

Finally, we will need the following simple lemma on immersions. Recall
that $\alpha_h$ is a positive contact form iff
\begin{equation}\label{eq:hh}
   h_1'h_2-h_2'h_1>0. 
\end{equation}

\begin{lemma}\label{lem:adj}
Let $h=(h_1,h_2):[-\delta,\delta]\to\C$ be an immersion satisfying the
contact condition~\eqref{eq:hh} and $h_1'>0$. Then:

(a) For any constant $B\geq 0$ the immersion $h+(0,B)$
satisfies~\eqref{eq:hh}. 

(b) If in addition $h_2(0)>0$, then for every $\eps>0$ there exists a
constant $A\in \R$ such that for all $t\in [0,1]$ the immersion 
$h+t(A,0)$ satisfies~\eqref{eq:hh} near $r=0$, and
$(h_1+A)(0)\in (0,\eps)$. 
\end{lemma}

\begin{proof}
(a) Condition~\eqref{eq:hh}, $h_1'>0$ and $B\geq 0$ imply
$$
   h_1'(h_2+B)-h_2'h_1 = (h_1'h_2-h_2'h_1) + h_1'B > 0.
$$
(b) At $r>0$ we have $h_1'h_2>0$ by assumption, so for sufficiently
small $\delta\in(0,\eps)$
$$
   0 < h_1'h_2 - h_2'\delta = h_1'h_2-h_2'(h_1+A)
$$
with $A:=\delta-h_1(0)$. Linear interpolation from the right hand side
to \\$h_1'h_2-h_2'h_1>0$ yields $h_1'h_2-h_2'(h_1+tA)>0$ for all
$t\in[0,1]$. 
\end{proof}

After these preparations, we now turn to the 

\begin{proof}[Proof of Theorem~\ref{thm:Giroux}] 
Let $(\om,\lambda)$ be a SHS obtained after application of
Corollary~\ref{cor:structure2}. We will use the following terminology:
$N^\pm$ and $N^0$ are called positive/negative contact parts and the
flat parts; connected components of $N^\pm$ and $N^0$ will be called
regions. We will construct the stable homotopy and the supporting open
book successively on the different types of regions. 

{\bf Flat regions. }
Consider a flat region $N$ with the primitive integer $1$-form
$\bar\lambda$ provided by Corollary~\ref{cor:structure2}. 
Integration of $\bar\lambda$ over paths from a fixed base point yields
a fibration
\begin{equation}\label{eq:fibflat}
   \pi:N\to S^1 
\end{equation}
such that $d\pi=\bar\lambda|_N$. Let $k_i\gamma_i$ denote the
restriction of the cohomology class $[\bar\lambda|_N]$ to the boundary
component $T_i$ of $N$, where $\gamma_i\in H^1(T_i;\Z)$ is a primitive
integer cohomology class and $k_i\in\N$ is the multiplicity. Since
$\bar\lambda$ is $T^2$-invariant near $T_i$, there exist coordinates
near $T_i$ in which the projection is given by
$\pi(r,\phi,\theta)=k_i\theta$.  

{\bf Contact regions. }
Now let $N$ be a contact region, so $\om=c\,d\lambda$ on $N$ for some
constant $c\neq 0$. After possibly switching the orientation of $N$,
we may assume that $\lambda$ is a {\em positive} contact form (but $c$
may still be negative). 
Let $[-\delta,\delta]\times T^2$ be a tubular neighbourhood of
a boundary component $T_i=\{0\}\times T^2$ of $N$ which is contained
in an integrable region (such that $[0,\delta]\times T^2\subset N$). 
On this neighbourhood the SHS $(\om,\lambda)$ is given by
$(c\,d\alpha_h,\alpha_h)$ for some immersion $h:[-\delta,\delta]\to\C$ 
satisfying the contact condition~\eqref{eq:h}. 
After a perturbation of $h$ supported near $r=0$ we may assume that
the cohomology class $[\alpha_h|_{T_i}]\in H^1(T_i;\R)$ is rational. Let
$\gamma_i$ be the primitive integer cohomology class in $H^1(T_i;\Z)$
positively proportional to $[\alpha_h|_{T_i}]$. We choose linear
coordinates $(\phi,\theta)$ on $T^2$ in which
$\gamma_i=[d\theta]$. Then the positive contact condition yields  
$d\theta\wedge d\alpha_h>0$ near $r=0$. Since
$\alpha_h=h_1(r)d\phi+h_2(r)d\theta$, the immersion
$h=(h_1,h_2)$ satisfies $h_1(0)=0$, $h_1'(0)>0$, and $h_2(0)>0$. So
after a perturbation of $h$ supported near $r=0$ we may assume that
$h=(r,a)$ near $r=0$ with a constant $a>0$. By a further
deformation (keeping the contact condition) supported near
$r=0$ we can achieve that $h=(r,1)$ near $r=0$. After performing these
deformations near all boundary components of $N$, we can apply the easy
case of Proposition~\ref{prop:relGiroux} in which condition (v) (with
$l_i=0$) holds near all boundary components. It yields a homotopy of
positive contact forms $\lambda_t$ on $N$ such that
$\lambda_t=\lambda$ near $\p N$, and $\lambda_1$ is supported by an
open book $(B,\pi)$ with $\pi(r,\phi,\theta)=n_i\theta$ near each
boundary component $T_i$. We obtain a corresponding stable homotopy
$(\om_t,\lambda_t)$ by setting $\om_t:=c\,d\lambda_t$ with the
constant $c\neq 0$ from above.  

{\bf Introducing small contact regions. }
It remains to consider an integrable region $[a,b]\times T^2$. The
open book projections on the adjacent contact/flat regions constructed
above provide primitive integer cohomology classes $\gamma_a\in
H^1(T^2,\Z)$ near $a$ and $\gamma_b\in H^1(T^2,\Z)$ near $b$. Note
that in general we will have $\gamma_a\neq\gamma_b$, in which case the
open book projections $\pi$ given near the boundary of $[a,b]\times
T^2$ do not extend over $[a,b]\times T^2$. To deal with this, we
choose a subdivision  
\begin{equation}\label{eq:subdivv}
   a=r_0<r_1<\dots<r_n=b
\end{equation}
and a sequence $\{\gamma_k\}_{k=1,\dots,n}$ of primitive integer
cohomology classes in $H^1(T^2,\Z)$ with the following properties:  
\begin{itemize}
\item $\gamma_1=\gamma_a$ and $\gamma_{n}=\gamma_b$; 
\item on the interval $[r_{k-1},r_{k}]$ ($k=1,\dots,n$) we have 
$\bar\gamma_k\wedge \om>0$, where $\bar\gamma_k=p_kd\phi+q_kd\theta$
is the $T^2$-invariant representative of $\gamma_k$.  
\end{itemize}
Recall that, according to Corollary~\ref{cor:structure2}, the
proportionality factor $f=d\lambda/\om$ is nowhere zero on
$[a,b]\times T^2$. We apply Proposition~\ref{prop:thick} (with
$Z=\{a\}=\{a_k\}$ and followed by averaging) to the level sets
$a_k=f(r_k)$, $k=1,...,n-1$, to find a new 
$T^2$-invariant stabilizing 1-form $\tilde\lambda$ for $\om$ such
that $\tilde f:=d\tilde\lambda/\om$ satisfies $\tilde f(r)\equiv
a_k\neq 0$ on some intervals $[r_k-\delta_k,r_k+\delta_k]$. We rename
$\tilde\lambda,\tilde f$ back to $\lambda,f$. We will refer to the
regions $\{r\in[r_k-\delta_k,r_k+\delta_k]$ as {\em small contact
  regions} in order to distinguish them from the original ({\em
  large})  
contact or flat regions constructed above. It is important to note
that adjacent small contact regions have the same sign, i.e.~the
contact structures are either both positive or both negative. 

We apply Proposition~\ref{prop:relGiroux} to construct contact
homotopies and suporting open books on the small contact regions. 
To extend them over the integrable regions $[r_{k-1},r_k]\times
T^2$, we distinguish two cases: integrable regions $[a,r_1]\times T^2$
and $[r_{n-1},b]\times T^2$ connect a large contact/flat region to a
small contact region, and regions $[r_{k-1},r_{k}]\times T^2$,
$k=2,...,n-1$, connecting two small contact regions. Recall that on
each such region we have $\bar\gamma_k\wedge\om>0$ for some
$T^2$-invariant 1-form on $T^2$ representing a primitive integer
cohomology class. After a linear change of coordinates we may assume
that $\bar\gamma_k=d\theta$ and thus 
$$
   d\theta\wedge\om>0. 
$$

{\bf Integrable regions I. }
Consider an integrable region $[a,r_1]\times T^2$ connecting a large
contact/flat region $N$ with a small contact region
$[r_1-\delta,r_1+\delta]\times T^2$. (The region
$[r_{n-1},b]\times T^2$ can be treated analogously). 

In both the contact and flat case, the stable homotopy on $N$
constructed above was constant near $\{r=a\}$. So in order to
extend the homotopy over $[a,r_1]$ we need to arrange rotation number
zero at $r=r_1-\delta$ in the application of 
Proposition~\ref{prop:relGiroux} to the small contact region. 
To achieve this, we prepare the stabilizing 1-form $\lambda$ before
applying Proposition~\ref{prop:relGiroux}. 

We write $\om=d\alpha_h$ and $\lambda=\lambda_g$
for an functions $h=(h_1,h_2):[a,b]\to \C$ and 
$g=(g_1,g_2):[a,b]\to \C$ satisfying~\eqref{eq:g}. 
We may assume that $\lambda$ is a positive contact form on
$[r_1-\delta,r_1+\delta]\times T^2$. (Otherwise we make the
orientation reversing coordinate change $\Psi:\phi\mapsto -\phi$ and
replace $(\om,\lambda)$ by $(-\Psi^*\om,\Psi^*\lambda)$).  
So we have $h(r)-cg(r)\equiv\const\in\C$ on $[r_1-\delta,r_1+\delta]$  
for some $c>0$. 

Fix some $\eps>0$ that will be specified later. 
We choose $B\ge 0$ such that $g_2(0)+B>0$. By Lemma~\ref{lem:adj} (a)
the homotopy $\{g+t(0,B)\}_{t\in [0,1]}$ satisfies the contact
condition on $[r_1-\delta,r_1+\delta]$. By Lemma~\ref{lem:adj} (b) we
find $A\in \R$ and a homotopy $\{g+(0,B)+t(A,0)\}_{t\in [0,1]}$ of
contact immersions on $[r_1-\delta,r_1+\delta]$ (for a possibly smaller
$\delta$) such that $(g_1+A)(r)\in(0,\eps)$ for all $r$. Denote by $g^t$,
$t\in[0,1]$, the concatenation of these two contact homotopies on
$[r_1-\delta,r_1+\delta]$ and note that $c(g^t)'=h'$ for all $t$, so the
$g^t$ satisfy~\eqref{eq:g} for all $t$. We use
Proposition~\ref{prop:stabhom} to extend this homotopy to $[a,b]$ such
that it satisfies~\eqref{eq:g} and coincides with $g$ outside a
neighbourhood of $[r_1-\delta,r_1+\delta]$. We rename the new
stabilizing function $g^1$ back to $g$, so we have achieved that $g$
is a contact immersion on $[r_1-\delta,r_1+\delta]$ with first
component $g_1\in(0,\eps)$. Moreover, we still have
$h(r)-cg(r)\equiv\const\in\C$ on $[r_1-\delta,r_1+\delta]$ with $c>0$
as above. (Note that we may have destroyed positivity of
$f=d\lambda/\om$, but we will not need this any more). 

Now we apply Proposition~\ref{prop:relGiroux} to the small contact
region $(N,\alpha)=([r_1-\delta,r_1+\delta]\times T^2,\lambda_g)$ to
find an open book decomposition $(B,\pi)$ and a contact homotopy
$\lambda_t$. Write $\lambda_t=\lambda_{g^t}$ near the boundary torus
$T_i=\{r_1-\delta\}\times T^2$. Proposition~\ref{prop:relGiroux} (iv)
(with $h^t=g^t$ and $l_i=r_1-\delta$) and $g_1\in(0,\eps)$ implies
that the first component of $g^1$ satisfies $g^1_1\in(0,\eps)$. 
Moreover, we can choose $(g^t)'(r_1-\delta)$, $t\in[0,1]$, to have
rotation number zero. 

We extend $\lambda_t$ to a stable homotopy $(\om_t,\lambda_t)$ on
$[r_1-\delta,r_1+\delta]\times T^2$ by $\om_t:=cd\lambda_t$. Near
$r=r_1-\delta$ we have $\om_t=\om_{h^t}$ with immersions $h^t$,
$t\in[0,1]$, defined by $h^t(r)-cg^t(r)\equiv\const\in\C$. Thus the
first components satisfy $|h_1^1-h_1| = c|g_1^1-g_1| <
c\eps$ near $r=r_1-\delta$. Since $h_1'>0$ on $[a,r_1]$, for
sufficiently small $\eps$ we can extend $h^1$ from a neighbourhood of
$r_1-\delta$ to an immersion on $[a,r_1-\delta]$ which coincides with
$h$ near $r=a$ and whose first component satisfies $(h_1^1)'>0$
everywhere. Moreover, since $(h^t)'(r_1-\delta)$, $t\in[0,1]$, has
rotation number zero, we can extend the $h^t$ from a neighbourhood of
$r_1-\delta$ to immersions on $[a,r_1-\delta]$ which coincide with
$h$ near $r=a$ and satisfy $h^0=h$. Now we use
Proposition~\ref{prop:stabhom} to extend the $g^t$ from a
neighbourhood of  $r_1-\delta$ to functions on $[a,r_1-\delta]$ which
coincide with $g$ near $r=a$ such that $g^0=g$ and $(h^t,g^t)$
satisfies~\eqref{eq:g} for all $t\in[0,1]$. Thus
$(\om_t,\lambda_t)=(d\alpha_{h^t},\lambda_{g^t})$ is a stable homotopy
on $[a,r_1-\delta]\times T^2$, which coincides with the previously
constructed homotopies near the boundary, from
$(\om_0,\lambda_0)=(\om,\lambda)$ to $(\om_1,\lambda_1)$ satisfying
$d\theta\wedge \om_1>0$.  

{\bf Integrable regions II. }
It remains to consider an integrable region $[r_{k-1},r_k]\times T^2$
connecting two small contact regions. Recall that the two contact
regions have the same sign, so we may assume that they are both
positive. (Otherwise we make the orientation reversing 
coordinate change $\Psi:\phi\mapsto -\phi$ and replace $(\om,\lambda)$
by $(-\Psi^*\om,\Psi^*\lambda)$). 

Let us ignore for the moment the cohomology classes of $\om_t$, which
we will discuss below. Then we conclude the argument as follows. 
We apply Proposition~\ref{prop:relGiroux} to both small contact
regions, choosing the option ``rotation number $-1$'' in
Proposition~\ref{prop:relGiroux} (iv) at both boundary tori
$T_{k-1}=\{r_{k-1}+\delta\}\times T^2$ and $T_k=\{r_k-\delta\}\times
T^2$. Note that the chosen coordinates $(r,\phi,\theta)$ on
$[r_{k-1},r_k]\times T^2$ are related to the coordinates in
Proposition~\ref{prop:relGiroux} by the identity map near $T_{k-1}$,
and (for example) by the map $(r,\phi,\theta)\mapsto
(-r,-\phi,\theta)$ near $T_k$. Thus, in the integrable region 
coordinates $(r,\phi,\theta)$ the rotation number equals $-1$ at
$T_{k-1}$ and $+1$ at $T_k$. (Here it is crucial that no positive and
negative contact regions are connected by an integrable region!). 
Thus, the homotopy of Hamiltonian structures $\om_t$ given on the
small contact regions extends over the region $[r_{k-1},r_k]\times
T^2$ (by extending the corresponding immersions $h^t$) such that
$\om_0=\om_1=\om$. We use Proposition~\ref{prop:stabhom} extend the
$\lambda_t$ from the contact regions to stabilizing 1-forms over
$[r_{k-1},r_k]\times T^2$. Note that
$d\theta\wedge\om_1=d\theta\wedge\om_0>0$ on $[r_{k-1},r_k]\times
T^2$.  

{\bf Defining the open book. }
The preceding step finishes the construction of the homotopy of SHS 
$(\om_t,\lambda_t)$, $t\in[0,1]$, on $M$. It remains to define the
open book $(B,\pi)$ supporting $(\om_1,\lambda_1)$. For this, recall
that in Proposition~\ref{prop:relGiroux} (ii) we have the freedom to
prescribe the multiplicities $n_i\geq K$ of the open book projection
near each boundary component $T_i$ of a contact region, for some
constant $K$ depending on this region. Let $K_0$ be the maximum of
these constants $K$ over all (large and small) contact regions.  
We set the multiplicities at all boundary tori of contact regions
equal to $K_0$, except for those adjacent to flat regions. 
For the latter the choice can be made as follows. Let
$[a,b]\times T^2$ be an integrable region with $T_a=\{a\}\times T^2$ 
contained in a flat region $N$ and $\{b\}\times T^2$ in a (small)
contact region $N^c$ (the opposite case is analogous). By the
discussion following equation~\eqref{eq:fibflat}, there is a primitive
cohomology class $\gamma_a\in H^1(T_a;\Z)$ and a multiplicity
$k_a\in\N$ such that $k_al_a=[d\pi|_{T_a}]$ for the projection
$\pi:N\to S^1$. We choose the multiplicity for
$T_b$ in Proposition~\ref{prop:relGiroux} to be the product $k_aK_0$.
Then the fibration $K_0\pi:N\to S^1=\R/\Z$ extends over the integrable
region and coincides with the one produced by
Proposition~\ref{prop:relGiroux} near $T_b$. With these choices, the
open book projections on the contact regions and the projections on
the flat regions extend over all the integrable regions to an open
book structure on $M$ supporting $(\om_1,\lambda_1)$. 

{\bf Preserving the cohomology class. }
The discussion so far finishes the proof of
Theorem~\ref{thm:Giroux} except for vanishing of the cohomology
classes $[\om_t-\om]\in H^2(M;\R)$. Let us analyze how we can ensure 
exactness in the previous constructions. 
On a flat region we have $\om_t=\om$. On a (large or small) contact
region $N$ we have $\om=c\,d\lambda$ for some constant $c\neq 0$. We
change $\lambda$ by a contact homotopy $\lambda_t$ and define $\om_t$
by $\om_t=c\,d\lambda_t$. 
Next consider an integrable region $U=[a,b]\times T^2$ connecting two
contact/flat regions $N_0,N_\ell$ on which $\om=c_id\lambda$
for constants $c_0,c_\ell\neq 0$. Moreover,
suppose that we have several small integrable regions
$N_i=[r_i-\delta,r_i+\delta]\times T^2$ in $U$ on which $\om=c_id\lambda$
for constants $c_i\neq 0$, $i=1,\dots,\ell-1$. We change $\lambda$ by
a contact homotopy $\lambda_t$ and define $\om_t$ by
$\om_t=c_id\lambda_t$ on each $N_i$, $i=0,\dots,\ell$. Denote by
$\p^\pm N_i=\{r_i\pm\delta\}\times T^2$ the left resp.~right boundary
components for $i=1,\dots,\ell-1$ and set $\p^+N_0=\{a\}\times T^2$,
$\p^-N_\ell=\{b\}\times T^2$. On $U$ we can write 
$\lambda=\lambda_{g}$ and $\om=d\alpha_{h}$ for functions
$g,h:[a,b]\to\C$ satisfying~\eqref{eq:g}. So near each boundary
component $\p^\pm N_i$ we have a relation
\begin{equation}\label{eq:const}
   h-c_i g = k_i^\pm
\end{equation}
for some constants $k_i^\pm\in\C$. Moreover, near $\p^\pm N_i$ we can
write $\lambda_t=\lambda_{g^t}$ for locally defined functions $g^t$. 
Now suppose that we find a family of immersions $h^t:[a,b]\to\C$ and a
family of constants $k^t\in\C$ such that $h^0=h$, $k^0=0$ and 
\begin{equation}\label{eq:exact}
   h^t = 
   \begin{cases}
      c_ig^t + k_i^\pm & \text{ near }r=r_i\pm\delta \\
      c_0g^t + k_0^+ + k^t & \text{ near }r=a \\
      c_\ell g^t + k_\ell^- + k^t & \text{ near }r=b. 
   \end{cases}
\end{equation}
Then we obtain an {\em exact} homotopy of HS $\om_t=\om+d\beta_t$ on
$U$, extending the given one on $\cup_iN_i$ and $T^2$-invariant on
$\cup_iN_i$, by setting 
\begin{equation*}
   \beta_t := 
   \begin{cases}
      c_0(\lambda_t-\lambda) & \text{ on }N_0 \\
      c_\ell(\lambda_t-\lambda) & \text{ on }N_\ell \\
      c_i(\lambda_t-\lambda) - \alpha_{k^t} & \text{ on }N_i,\
      i=1,\dots,\ell-1 \\ 
      \alpha_{h^t} - \alpha_h - \alpha_{k^t} & \text{ on
      }U\setminus\cup_iN_i. 
   \end{cases}
\end{equation*}
Finally, Proposition~\ref{prop:stabhom} provides a homotopy of
stabilizing 1-forms $\lambda_t$ for $\om_t$ on $U$, extending the
given one on $\cup_iN_i$ and $T^2$-invariant on $\cup_iN_i$. 
Note that in ``Integrable regions I'' above we arranged
condition~\eqref{eq:exact} with constant $k^t=0$, but in ``Integrable
regions II'' this will in general fail. 

{\bf The $\theta$ -- $\phi$ -- $\theta$ trick. }
Consider again an integrable region $[a,b]\times T^2$ connecting two
(small) positive contact regions $N_a,N_b$ with a SHS
$(\om,\lambda)=(d\alpha_h,\lambda_g)$ satisfying $d\theta\wedge\om>0$,
i.e.~$h_1'>0$. Near $r=a,b$ we are given a contact homotopy
$\lambda_t$ and define $\om_t$ by $\om_t=c_ad\lambda_t$
resp.~$\om_t=c_bd\lambda_t$, i.e.~  
$$
      h(r)-c_a g(r) = k_a \text{ resp. }h(r)-c_b g(r) = k_b
$$
for constant $c_a,c_b>0$ and $k_a,k_b\in\C$. We face the following
problem: In order to ensure exactness we wish to define the homotopy
$h^t:[a,b]\to\C$ satisfying~\eqref{eq:exact} near $r=a,b$, but then
the difference of the first components of 
$h^1(b)-h^1(a)=c_bg^1(b)+k_b-c_ag^1(a)-k_a$ may be negative and we
cannot achieve $(h^1_1)'>0$. To overcome this problem, we introduce
two new small contact regions $N_x=[x-\delta,x+\delta]\times T^2$ and
$N_y=[x-\delta,y+\delta]\times T^2$ around new subdivision points $x$ and $y$: 

$$
   a<x<y<b.
$$ 
Now we repeat the construction using the following open book
projections: 
\begin{equation}\label{eq:proj}
   \pi(r,\phi,\theta) := 
   \begin{cases}
      \theta & \text{ on }[a,x-\delta] \\
      \phi & \text{ on }[x+\delta,y-\delta] \\
      \theta & \text{ on }[y+\delta,b].
   \end{cases}
\end{equation}
It will turn out that, with this trick, we can use the freedom in
Proposition~\ref{prop:relGiroux} and Lemma~\ref{lem:adj} to achieve
exactness of $[\om_t-\om]$ as well as positivity of
$d\theta\wedge\om_1$ resp.~$d\phi\wedge\om_1$ on the respective
regions. This will be carried out in the remaining two steps. 

{\bf Preparation for the exact homotopy. }
We first homotop the immersion $h:[a,b]\to \C$ through immersions rel
$\p[a,b]$ to one (still denoted by $h$) which satisfies
$$
   h(r) = e^{-i(r-x+\pi/4)}
$$
on some small subinterval $[x-\delta,y+\delta]\subset(a,b)$, see
Figure \ref{figure:nose}. 

\begin{figure}
\centering
\includegraphics{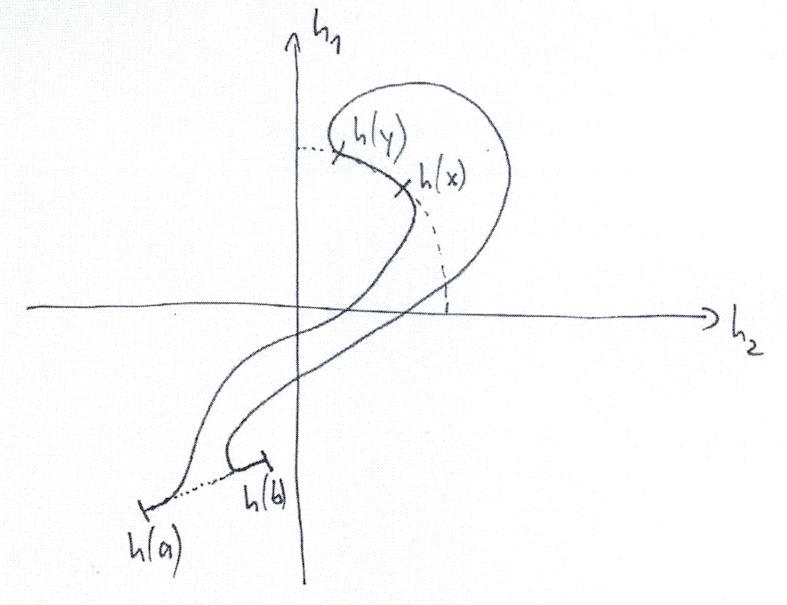}
\caption{The deformation of $h$}
\label{figure:nose}
\end{figure}

Hence $h|_{[x-\delta,y+\delta]}$ is contact and satisfies the conditions 
$$
   \hat h_1'>0,\qquad \hat h_2'<0,\qquad ih'=h. 
$$
Note that we may have lost the condition $(h_1)'>0$ outside the
interval \\$[x-\delta,y+\delta]$. The 1-form
$h_1(r)d\phi+h_2(r)d\theta$ stabilizes $d\alpha_h$ on
$[x-\delta,y+\delta]$. So by  
Proposition~\ref{prop:stabhom} we find a stabilizing form $\lambda_g$
for $\om=d\alpha_h$ on $[a,b]\times T^2$ which agrees with the
original $\lambda$ near the boundary and such that $g=h$ on
$[x-\delta,y+\delta]$.  

Next we use Lemma~\ref{lem:adj} and Proposition~\ref{prop:stabhom} to
modify the stabilizing 1-form $\lambda_g$ on the small contact regions
$N_x=[x-\delta,x+\delta]\times T^2$ and $N_y=[y-\delta,y+\delta]\times
T^2$ (possibly shrinking $\delta$ in the process).  
For this, we make the following orientation preserving coordinate
change on $[x-\delta,y+\delta]\times T^2$:
$$
   \hat\theta:=\phi,\quad \hat\phi:=-\theta.
$$
We rewrite $\lambda_g$ in the new coordinates,
$$
   \lambda_g = g_1d\phi+g_2d\theta = \hat g_1d\hat\phi+\hat
   g_2d\hat\theta,\qquad \hat g_1=-g_2,\quad \hat g_2=g_1
$$
and note that the contact condition $\la ig',g\ra>0$ is preserved by
this coordinate change. Moreover, by construction we have $(\hat
g_1)'=-g_2'=-h_2'>0$. Thus, by Lemma~\ref{lem:adj} (applied in the
coordinates $(\hat\phi,\hat\theta)$ near $x$ and $y$) and
Proposition~\ref{prop:stabhom}, we can add
complex constants to $g$ near $x,y$ to obtain a new stabilizing form,
still denoted by $\lambda_g$, which coincides with the previous one
outside a neighbourhood of $N_x\cup N_y$ and satisfies
\begin{equation}\label{eq:prep}
\begin{cases}
   0<\hat g_1<\eps \text{ and }g_1=\hat g_2>0 &\text{ on
   }[x-\delta,x+\delta], \\
   0<\hat g_1<\eps \text{ and }g_1=\hat g_2\geq A>0 &\text{ on
   }[y-\delta,y+\delta]
\end{cases}
\end{equation}
for an arbitrarily small constant $\eps$ and an arbitrarily large
constant $A$ that will be specified later. 

{\bf Constructing the exact homotopy. }
Now we apply Proposition~\ref{prop:relGiroux} to the contact regions
$N_a$, $N_b$, $N_x$ and $N_y$, with the open book projections on the
connecting integrable regions given by $\theta$
resp.~$\hat\theta=\phi$ as in~\eqref{eq:proj}. We choose the option
``rotation number -1'' in Proposition~\ref{prop:relGiroux} (iv) at all
boundary tori.
Recall that near the boundary component $\p^\pm N_x,\p^\pm
N_y,N_a,N_b$ we have the respective relations
\begin{equation*}
   h - g = k_x^\pm,\quad 
   h - g = k_y^\pm,\quad 
   h - c_ag = k_a,\quad 
   h - c_bg = k_b 
\end{equation*}
for some constants $c_a,c_b>0$ and
$k_x^\pm,k_y^\pm,k_a,k_b\in\C$. Moreover, near each boundary torus the
homotopy of contact forms $\lambda_t$ from
Proposition~\ref{prop:relGiroux} can be written as $\lambda_{g^t}$ for
locally defined functions $g^t$. We define a family of immersions
$h^t$ near the boundary tori as in~\eqref{eq:exact} with $k^t=-t(K,0)$,
i.e.~  
\begin{equation}\label{eq:exact2}
   h^t = 
   \begin{cases}
      g^t + k_x^\pm & \text{ near }r=x\pm\delta \\
      g^t + k_y^\pm & \text{ near }r=y\pm\delta \\
      c_ag^t + k_a - t(K,0) & \text{ near }r=a \\
      c_bg^t + k_b - t(K,0) & \text{ near }r=b,
   \end{cases}
\end{equation}
where $K\geq 0$ is a real constant that will be chosen below. 
We need to extend $h^t$ to a family of immersions $[a,b]\to\C$ 
such that $h^0=h$ and $h^1$ is supported by the respective open books,
i.e.~the $d\phi$-component $h^1_1$ increases on
$[a,x-\delta]\cup[y+\delta,b]$ and the $d\hat\phi$-component $\hat
h^1_1=-h^1_2$ increases on $[x+\delta,y-\delta]$. Note that, once we
have constructed $h^1$ with these properties, the extension of the
homotopy $h^t$ follows as in ``Integrable regions II'' from our choice
of rotation numbers. 

First, we consider the interval $[x+\delta,y-\delta]$ in coordinates
$(\hat\phi,\hat\theta)$.  Near $x+\delta$ and $y-\delta$ we
have $\hat g_1\in(0,\eps)$ according to~\eqref{eq:prep}, hence $\hat
g^1_1\in(0,\eps)$ by Proposition~\ref{prop:relGiroux} and $\hat
h^1_1-\hat h_1\in(0,\eps)$ by~\eqref{eq:exact2}. Since $\hat h_1=-h_2$
strictly increases on $[x,y]$, it follows that $\hat
h^1_1(y-\delta)>\hat h^1_1(x+\delta)$, so we can extend the immersion
$h^1$ over $[x+\delta,y-\delta]$ so that $\hat h^1_1$ is strictly
increasing. 

Next, we consider the interval $[a,x-\delta]$ in coordinates
$(\phi,\theta)$. According to~\eqref{eq:exact2}, the $d\phi$-component
of $h^1$ changes over this interval by the amount
$$h^1_1(x-\delta)-h^1_1(a)=C+K\in\R,$$ where 
$$
   C := g^1_1(x-\delta)+(k_x^-)_1-c_ag^1_1(a)-(k_a)_1 \in \R 
$$
depends on the functions $h,g$ constructed above near $a$ and
$x-\delta$  (but not on the constant $A$ in~\eqref{eq:prep}!). We
choose the constant $K\geq 0$ in~\eqref{eq:exact2} large enough such
that $C+K>0$, so we can extend the immersion $h^1$ over
$[a,x-\delta]$ with $h^1_1$ is strictly increasing.  

Finally, we consider the interval $[y+\delta,b]$ in coordinates
$(\phi,\theta)$. Recall that near the boundary torus $\{r=y+\delta\}$
the coordinates $(r,\phi,\theta)$ on the integrable region $[a,b]\times
T^2$ are related to the coordinates of
Proposition~\ref{prop:relGiroux} by the coordinate change
$(r,\phi,\theta)\mapsto (-r,-\phi,\theta)$.    
In the coordinates of Proposition~\ref{prop:relGiroux} the $d\phi$-
component of the stabilizing form $\lambda$ at $y+\delta$ changes from
the (very negative) value $-g_1(y+\delta)\leq -A$ (here we
use~\eqref{eq:prep}!) to a positive value $-g^1_1(y+\delta)>0$, so the
value increases by at least $A$. Switching back to the integrable
region coordinates $(r,\phi,\theta)$, we see a decrease by at least
$A$:  
\begin{equation}\label{eq:y}
   g^1_1(y+\delta)-g_1(y+\delta) < -A.
\end{equation}
The $d\phi$-component of $h^1$ changes over the interval
$[y+\delta,b]$ by the amount 
$$
   h^1_1(b)-h^1_1(y+\delta) 
   = [h^1_1(b)-h_1(b)] + [h_1(b) - h_1(y+\delta)] + [h_1(y+\delta) -
   h^1_1(y+\delta)]. 
$$
By~\eqref{eq:exact2} the first term in $[\ ]$ on the right hand side
equals 
$$
   h^1_1(b)-h_1(b) = c_b\bigl(g^1_1(b)-g_1(b)\bigr) - K, 
$$
which depends on the functions $h,g$ constructed above near $b$ and
the constant $K$ chosen above, but not on the constant $A$. 
The second term depends on the function $h$ near $y+\delta$ and $b$,
but not on $g$ and hence not on the constant $A$. The third term is
estimated using~\eqref{eq:exact2} and~\eqref{eq:y} by
$$
   h_1(y+\delta) - h^1_1(y+\delta) = g_1(y+\delta) - g^1_1(y+\delta) >
   A. 
$$
Therefore, by choosing the constant $A\geq 0$ large enough we can
achieve that $h^1_1(b)-h^1_1(y+\delta)>0$, so we can extend the
immersion $h^1$ over $[y+\delta,b]$ with $h^1_1$ is strictly increasing. 
This concludes the proof of Theorem~\ref{thm:Giroux}.
\end{proof}


\end{document}